\documentclass[11pt]{article}
\usepackage[a4paper,margin=25mm]{geometry}
\usepackage[T1]{fontenc}
\usepackage{lmodern,amsmath,amssymb,amsthm,mathtools,microtype,booktabs,array}
\usepackage[hidelinks]{hyperref}
\usepackage{aliascnt}
\usepackage[nameinlink,noabbrev]{cleveref}
\numberwithin{equation}{section}
\newtheorem{theorem}{Theorem}[section]
\newaliascnt{proposition}{theorem}
\newtheorem{proposition}[proposition]{Proposition}
\aliascntresetthe{proposition}
\newaliascnt{lemma}{theorem}
\newtheorem{lemma}[lemma]{Lemma}
\aliascntresetthe{lemma}
\newaliascnt{corollary}{theorem}
\newtheorem{corollary}[corollary]{Corollary}
\aliascntresetthe{corollary}
\theoremstyle{definition}\newaliascnt{definition}{theorem}
\newtheorem{definition}[definition]{Definition}
\aliascntresetthe{definition}
\theoremstyle{remark}\newaliascnt{remark}{theorem}
\newtheorem{remark}[remark]{Remark}
\aliascntresetthe{remark}
\crefname{theorem}{Theorem}{Theorems}
\crefname{lemma}{Lemma}{Lemmas}
\crefname{proposition}{Proposition}{Propositions}
\crefname{corollary}{Corollary}{Corollaries}
\newcommand{\PP}{\mathbb P}\newcommand{\C}{\mathbb C}\newcommand{\R}{\mathbb R}
\newcommand{\cC}{\mathcal C}\newcommand{\cZ}{\mathcal Z}\newcommand{\cT}{\mathcal T}
\newcommand{\Gr}{\operatorname{Gr}}
\newcommand{\Chow}{\operatorname{Chow}}\newcommand{\Hilb}{\operatorname{Hilb}}
\newcommand{\Vol}{\operatorname{Vol}}\newcommand{\diam}{\operatorname{diam}}
\newcommand{\dist}{\operatorname{dist}}\newcommand{\supp}{\operatorname{supp}}
\newcommand{\tr}{\operatorname{tr}}\newcommand{\Sing}{\operatorname{Sing}}\newcommand{\id}{\mathrm{id}}
\newcommand{\eps}{\varepsilon}
\title{Wasserstein Metric Normalization of Projective Cycle Spaces}
\author{Jiaping Yang\thanks{School of Mathematical Sciences, Fudan University, 220 Handan Road, Shanghai 200433, China} \thanks{Email:jpyang22@m.fudan.edu.cn}}
\date{}
\begin{document}
\maketitle

\begin{abstract}
Let $X$ be an irreducible reduced projective subvariety of a Chow variety.  On a
generic smooth parameter locus, normal motions of resolved components define a
quadratic Wasserstein metric, even when the cycles are singular, reducible, or
carry multiplicities.  Its metric completion is canonically $X^\nu$: resolution
fibers collapse, whereas distinct normalization branches over the same Chow cycle
remain separated.  The completed metric admits an exact ambient action formula.
For degree-$d$ hypersurfaces, the inner $W_q$ geometry is compact and induces the
projective topology for every $1\le q\le2$.  For $d\ge2$, the uniform H\"older
exponent $1/d$ is attained and sharp for each such $q$; when $d=1$, the comparison
is Lipschitz.  Elliptic quartics furnish an explicit model of the boundary
branching.
\end{abstract}

\medskip\noindent\textbf{Keywords.} Optimal transport; Wasserstein geometry; algebraic cycles; metric completion; normalization.

\medskip\noindent\textbf{MSC 2020.} Primary 49Q22, 14C05; Secondary 14B05, 32U05, 53C55.

\section{Introduction}

A smooth embedded projective cycle has two natural descriptions: it is a point of
a Chow variety, and its normalized Fubini--Study volume is a probability measure
on projective space.  For a smooth deformation $Z_t$ with normal velocity $N_t$,
\begin{equation}\label{eq:intro-speed}
  |\dot\mu_t|_{W_2}^2=\int_{Z_t}|N_t|^2\,d\mu_t.
\end{equation}
The identity induces the intrinsic quadratic Wasserstein length on the smooth cycle
locus.  Two boundary questions then become natural: which algebraic boundary does
this length select, and how much of the transport structure survives there?

Let
\[
  X\subset\Chow_{m,D}(\PP^n)
\]
be an irreducible reduced projective subvariety, and let
\[
  \nu:Y=X^\nu\longrightarrow X
\]
be its normalization.  On a dense smooth Zariski-open locus
$B\subset X_{\rm reg}$, the irreducible components of the family admit local
smooth proper resolutions with fixed multiplicities; see
\cref{def:admissible,lem:generic-resolved}.  The cycles parametrized by $B$ need
not themselves be smooth or irreducible, and may carry multiplicities.  The normal
motions of the resolved components define a smooth positive definite metric
$g_B$ satisfying
\begin{equation}\label{eq:intro-general-speed}
  |\dot\mu_{b(t)}|_{W_2}^2=g_B(\dot b(t),\dot b(t))
\end{equation}
for every absolutely continuous curve $b(t)$ in $B$.

\begin{theorem}[Metric normalization]\label{thm:intro-main}
The intrinsic Wasserstein completion of $B$ is canonically homeomorphic to $Y$:
\[
  \widehat{(B,d_B)}\cong X^\nu.
\]
The metric space $(Y,\bar d)$ is compact, complete, and geodesic.  Moreover, for
every chordal distance $h_Y$ induced by a projective embedding of $Y$, there exist
$C,\beta>0$ such that
\[
  \bar d(y,z)\le C h_Y(y,z)^\beta.
\]
Distinct normalization branches over one Chow cycle define distinct completion
points, whereas resolution data over a fixed normalization point collapse.
\end{theorem}

\Cref{thm:intro-main} is proved in \cref{sec:normalization}.  The
locus-independence assertion in the stronger formulation \cref{thm:main} is
derived later from the global action formula; see
\cref{cor:admissible-locus-independent}.

The metric completion also retains the ambient Wasserstein dynamics.  Define
\[
  \mathcal M:Y\longrightarrow\mathcal P_2(\PP^n),
  \qquad \mathcal M(y)=\mu_{\nu(y)}.
\]

\begin{theorem}[Ambient action identity]\label{thm:intro-action}
Let $\gamma:[0,T]\to Y$ be continuous and set $\mu_t=\mathcal M(\gamma(t))$.  For
every $1\le p\le\infty$,
\[
  \gamma\in AC^p([0,T],\bar d)
  \quad\Longleftrightarrow\quad
  \mu_\cdot\in AC^p([0,T],W_2),
\]
and
\[
  |\dot\gamma(t)|_{\bar d}=|\dot\mu_t|_{W_2}
  \quad\text{for a.e. }t.
\]
On $[0,1]$, the squared distance $\bar d(y_0,y_1)^2$ is the minimum ambient
Wasserstein action among continuous normalization paths joining $y_0$ to $y_1$.
\end{theorem}

The $AC^p$ equivalence and metric-speed identity are proved in
\cref{thm:full-action}, while the minimizing-action formula is established in
\cref{cor:action-distance}.  As further consequences, $\mathcal M$ preserves the
length of every continuous path and the completed metric is independent of the
admissible smooth locus.

The H\"older exponent in \cref{thm:intro-main} is controlled in general by the
singularity exponents of the Chow forms and by a \L{}ojasiewicz exponent of a
resolution; when $X^\nu$ is smooth, the latter loss disappears.  For
hypersurfaces the comparison is sharp: both the topology of the completion and
the optimal H\"older scale can be determined explicitly.  Let
\[
  P_{n,d}:=\PP H^0(\PP^n,\mathcal O_{\PP^n}(d)).
\]
For every $1\le q\le2$, the inner Wasserstein distance
$W_{q,\mathrm{ACL}}^{\rm in}$ induces the projective topology on $P_{n,d}$,
which is therefore compact for this metric.  For any chordal distance $h$, if
$d\ge2$,
\[
  W_{q,\mathrm{ACL}}^{\rm in}(p,p')\le C_q h(p,p')^{1/d},
\]
and $1/d$ is the optimal uniform exponent for every fixed $q\in[1,2]$.  When
$d=1$, the comparison is Lipschitz.  For $q=2$, the completed metric agrees with
the inner distance of \cite{ACL} up to the Fubini--Study normalization factor;
in particular, this resolves the $q=2$ compactness problem for the inner
Wasserstein metric posed in \cite{ACL}.  See
\cref{thm:acl-compactness,cor:acl-identification}.

Smooth complete intersections provide a basic class of examples, while elliptic
quartics make the boundary branching explicit: if the singular locus of an
irreducible plane quartic consists of exactly $\delta\ge2$ ordinary nodes and no
other singularities, then its Chow point has $\binom{\delta}{2}$ normalization
branches.

\subsection{Proof outline}

The proof is organized around a logarithmic Chow potential obtained by averaging
over the Grassmannian.  The incidence calculation
\begin{equation}\label{eq:intro-curvature}
  \Omega_B=\frac1{D(m+1)}\,\chi^*\cT
\end{equation}
identifies its curvature on the smooth locus, while its H\"older modulus controls
intrinsic distance near the boundary.  After passing to a resolution,
weak-gradient rigidity collapses resolution fibers, whereas the ambient $W_2$
lower bound keeps distinct normalization branches apart; together these facts
identify the metric completion.  To recover dynamics on the completed space, we
stratify the normalization and prove the exact speed identity on each stratum.
The remaining difficulty---curves that cross strata on highly fragmented time
sets---is handled by moment coordinates and a countable stitching argument.  For
hypersurfaces, a sharp sublevel estimate, supplemented in degree two by a
symmetric-root argument, yields the optimal H\"older exponent.

\subsection{Relation to previous work}

The complex-geometric background lies in the classical theory of parameter
spaces of cycles.  Barlet's construction of analytic cycle spaces and Fujiki's
closedness results for Douady spaces of compact K\"ahler spaces provide the
basic framework \cite{Barlet75,Fujiki78}; see also the modern treatment of
Barlet--Magn\'usson \cite{BM25}.  Barlet--Varouchas proved the holomorphic
variation of integration functionals on reduced cycle spaces
\cite{BarletVarouchas89}, a result closely related to the fiber-integration
formulas used here.  Axelsson--Schumacher developed generalized
Weil--Petersson geometry on Douady spaces and its extension as a positive
current across singular fibers \cite{AS06,AS21}.  On the projective side, Chow
forms and logarithmic Chow norms provide global potentials adapted to this
setting; we use the classical framework of \cite{GKZ,PS02}.  Analytically, the
distance estimates for K\"ahler currents with H\"older potentials in \cite{GGZ}
are close in spirit to the boundary estimates below.

The metric viewpoint comes from the dynamical and Riemannian formulations of
optimal transport.  The Benamou--Brenier action formula \cite{BB00}, Otto's
formal Riemannian calculus \cite{Otto01}, and the metric theory of
\cite{AGS} relate absolutely continuous Wasserstein curves to continuity
equations and minimal velocity fields; see also Lott's differential-geometric
calculations on Wasserstein space \cite{Lott08}.  In the present setting this
ambient geometry is restricted to a finite-dimensional family of cycle
measures.  The local formula is governed by normal motion of the cycle, while
the main global issue is what happens to the intrinsic length metric when the
algebraic family reaches a singular or nonnormal boundary.

The closest direct predecessor is the work of Antonini--Cavalletti--Lerario on
projective hypersurfaces \cite{ACL}.  They introduced the corresponding inner
Wasserstein distance, proved completeness and geodesicity on the full parameter
space, and identified a Weil--Petersson-type K\"ahler structure away from the
discriminant.  Their compactness theorem covers $q$ in a range above one, while compactness
for the quadratic inner metric $W_{2,\mathrm{ACL}}^{\rm in}$ was left open.  Our
concern is different: for a fixed
projective Chow subvariety, we ask which boundary is selected by the intrinsic
transport metric.  The resulting completion retains normalization
branches that the Chow point itself does not distinguish.  For hypersurfaces,
the construction specializes to the metric of \cite{ACL} and extends the
compactness/topology comparison through $1\le q\le2$.

A complementary direction appears in the joint work of Antonini, Cavalletti,
Cecchi, and Lerario announced in \cite{LerarioFoCM26}, which develops Wasserstein
geometry for algebraic and analytic cycle spaces and relates the regular locus to
Weil--Petersson-type structures.  The elliptic quartic model
used later comes from a separate compactification theory: we use the boundary
analysis of Avritzer--Vainsencher \cite{AV92}, the birational description of the
principal Hilbert component in \cite{GLS}, and the recent normal forms in
\cite{CKSZ} to make the branch phenomenon explicit.

\subsection{Organization}

Section~2 constructs the resolved Wasserstein metric and its Chow curvature.
Section~3 proves metric normalization and the quantitative hypersurface results.
Sections~4--5 establish the stratumwise and global action formulas, and
Section~6 treats the elliptic quartic boundary model.

\section{Resolved Wasserstein geometry}\label{sec:geometry-curvature}

Fix integers $0\le m<n$ and $D\ge1$, and let
\[
  X\subset\Chow_{m,D}(\PP^n)
\]
be irreducible and reduced.  Write
\[
  \nu:Y=X^\nu\longrightarrow X
\]
for the normalization.  The case $\dim X=0$ is immediate, so we assume
$\dim X>0$ below.  Normalize the Fubini--Study form by
$\int_{\PP^1}\omega=1$ and use $dd^c=(i/\pi)\partial\bar\partial$.  With this normalization,
$\int_{\PP^r}\omega^r=1$ and
\[
  V_m:=\Vol(\PP^m)=\frac1{m!}.
\]
For an effective cycle $C=\sum_j a_jC_j$ of dimension $m$ and degree $D$, with
$C_j$ distinct reduced irreducible components, define
\begin{equation}\label{eq:measure}
  \mu_C=\frac1{DV_m}\sum_j a_j\,d\Vol_{C_j^{\rm reg}}.
\end{equation}
The normalization makes $\mu_C$ a probability measure even when $C$ is singular,
reducible, or carries nontrivial multiplicities.

\begin{definition}[Admissible locus]\label{def:admissible}
A dense smooth Zariski-open subset $B\subset X_{\rm reg}$ is called
\emph{admissible} if every point has an analytic neighborhood $U\subset B$ on
which there are smooth proper holomorphic submersions
\[
  p_i:\widetilde Z_i\longrightarrow U
\]
with connected smooth $m$-dimensional fibers, holomorphic maps
\[
  e_i:\widetilde Z_i\longrightarrow\PP^n,
\]
and fixed integers $a_i>0$, such that for every $b\in U$
\begin{equation}\label{eq:generic-resolved-cycle}
  C_b=\sum_i a_i(e_i)_*[\widetilde Z_{i,b}].
\end{equation}
The images are the distinct reduced irreducible components of $C_b$, each
$e_{i,b}$ is birational onto its image, and the isomorphism loci may be chosen
as relative open subsets whose complements are fiberwise proper analytic sets.
\end{definition}

\begin{lemma}\label{lem:generic-resolved}
Let $W$ be an irreducible reduced complex algebraic variety and let
$\kappa:W\to\Chow_{m,D}(\PP^n)$ be finite onto its image.  There is a nonempty
smooth Zariski-open set $U\subset W$ on which $\kappa$ is immersive and for
which the cycle family $\kappa|_U$ admits resolved presentations of the form
\eqref{eq:generic-resolved-cycle}, locally in the analytic topology.  In
particular, every irreducible reduced projective subvariety
$X\subset\Chow_{m,D}(\PP^n)$ has an admissible dense smooth Zariski-open locus.
\end{lemma}

\begin{proof}
Let $K=\C(W)$ and consider the cycle over $\operatorname{Spec} K$ induced by $\kappa$.
After a finite separable extension $K'/K$, its geometrically irreducible reduced
components are defined over $K'$; denote them by $C_i$, with multiplicities
$a_i>0$.  Resolve each component projectively,
\[
  r_{i,\eta}:\widetilde C_i\longrightarrow C_i,
\]
and compose with $C_i\hookrightarrow\PP^n_{K'}$.  Let $g:W'\to W$ be the normalization of $W$ in $K'$.  Since $W$ is of finite
type over $\C$, it is Nagata, and normalization in the finite extension $K'/K$ is
finite; see
\cite[\href{https://stacks.math.columbia.edu/tag/035B}{Tag~035B};
\href{https://stacks.math.columbia.edu/tag/0AVK}{Tag~0AVK}]{Stacks}.
Thus $g$ is finite.  After replacing $W'$ by a nonempty open subset, the components, their resolutions, the
evaluation maps, and the birational isomorphism loci all spread out.  Generic flatness, together with openness of geometric reducedness and
irreducibility, lets us shrink further so that the component families have
integral $m$-dimensional fibers.  Removing the image of the nonsmooth locus of
the resolved families makes each $p_i$ smooth and projective.  Its generic
fiber is a resolution of a geometrically irreducible component, hence is
connected.  Shrinking the base if necessary, Stein factorization
\cite[\href{https://stacks.math.columbia.edu/tag/03H0}{Tag~03H0}]{Stacks} shows that the finite factor has degree one and
is therefore an isomorphism; all fibers of $p_i$ are then connected.  We further
arrange that the birational isomorphism loci are dense in every fiber.

Proper pushforward of the resolved fundamental cycles defines component Chow
maps $\kappa_i$.  Their weighted sum $\sum_i a_i\kappa_i$ agrees with $\kappa\circ g$ at the
generic point.  Since the Chow variety is separated, the two maps agree on the
chosen integral open set.  We remove the closed locus where distinct generic components collide.  If $F$
is the discarded subset of $W'$, replacing the base by the saturated open
$W\setminus g(F)$ preserves all preceding properties on the full inverse image.
Because $K'/K$ is separable, a further saturated shrinking makes
the finite map $W'\to W$ finite \'etale.  On sufficiently small simply connected
analytic neighborhoods it splits into holomorphic inverse branches; pulling
the resolved families back along one branch recovers the asserted local
presentation on the base itself.  A finite morphism in characteristic zero is generically unramified.  Intersecting
with the smooth locus of $W$ and with the full-rank locus of $d\kappa$ therefore
produces the required open set $U$.  Taking $W=X$ and $\kappa$ to be the
inclusion proves the last assertion.
\end{proof}

From now on, $B\subset X_{\rm reg}$ denotes an arbitrary admissible dense
smooth Zariski-open set.  Irreducibility of $X$ makes $B$ connected, and the
normalization is an isomorphism over $X_{\rm reg}$; we therefore identify $B$
with its inverse image in $Y$.

Convergence in a projective Chow space of fixed degree implies weak convergence
of integration currents; see \cite{Barlet75,BM25}.  Compactness of $\PP^n$
and \eqref{eq:measure} then show that
\[
  C\longmapsto\mu_C
\]
is continuous with respect to $W_2$.  It is also injective on effective cycles of fixed degree.  The support determines
the distinct irreducible components, and at a generic regular point of each
component the density relative to intrinsic volume recovers its integer
multiplicity once the total degree $D$ is fixed.

\subsection{H\"older Chow potentials}
Near the Chow boundary, the metric is controlled by a potential for $\Omega_B$
that remains regular on the compactification.  Its modulus of continuity comes
from a finite-dimensional logarithmic estimate.

Let $M$ be a compact connected complex manifold, $L\to M$ a holomorphic line
bundle with a smooth Hermitian metric, $d\lambda$ a smooth finite measure, and
$V\subset H^0(M,L)$ a finite-dimensional vector space.  Fix any norm on $V$.

\begin{lemma}\label{lem:negative}
For every compact $K\subset V\setminus\{0\}$ there exist $\eta>0$ and $C_K<\infty$
such that
\begin{equation}\label{eq:negative}
 \sup_{s\in K}\int_M\|s(x)\|^{-\eta}\,d\lambda(x)\le C_K.
\end{equation}
\end{lemma}
\begin{proof}
Fix $(s_0,x_0)\in K\times M$.  In a local holomorphic frame, either
$s_0(x_0)\ne0$ or, after a linear change of coordinates, the restriction of the
nonzero germ $s_0$ to the $z_1$-axis has some finite positive vanishing order
$a$.  Taking the coefficients of $s$ and the remaining variables $z'$ as
parameters, the parameter-dependent Weierstrass preparation theorem provides, on a
smaller polydisc, the factorization
\[
 s(z)=u(s,z)P_{s,z'}(z_1),
\]
where $u$ is a holomorphic unit and $P_{s,z'}$ is monic of degree $a$ in $z_1$.
On a smaller neighborhood, $u$ is uniformly bounded above and away from zero,
while the coefficients of $P_{s,z'}$ remain in a compact set.  For each fixed
$(s,z')$, write the multiset of roots of $P_{s,z'}$ as
$\{\zeta_1,\ldots,\zeta_a\}$, counted with multiplicity.  Their moduli are uniformly bounded by the coefficients,
and, if $a\eta<2$, H\"older's inequality in the $z_1$ variable shows that
\[
 \int_{|z_1|<r}\prod_{j=1}^a|z_1-\zeta_j|^{-\eta}\,dA(z_1)
 \le\prod_{j=1}^a\left(\int_{|z_1|<r}|z_1-\zeta_j|^{-a\eta}\,dA(z_1)\right)^{1/a}
 \le C.
\]
The bound is uniform as long as the roots stay in a fixed bounded set.  After
integrating in $z'$ and absorbing the bounded frame and volume factors, the same
uniformity remains.  A finite cover of $K\times M$ then supplies one sufficiently
small exponent $\eta>0$ valid throughout.  On the nonvanishing neighborhoods the section is uniformly
bounded away from zero, so the same estimate holds there for every $\eta>0$.
\end{proof}

Define the uniform complex singularity exponent of the finite-dimensional linear
system $V$ by
\begin{equation}\label{eq:uniform-cse}
 c_V:=\inf_{([s],x)\in\PP(V)\times M}c_x(s),
\end{equation}
where $c_x(s)$ is the complex singularity exponent, namely the supremum of $c>0$
for which $\|s\|^{-2c}$ is locally integrable near $x$.  The preceding lemma shows that $c_V>0$.  By the semicontinuity theorem of
Demailly--Koll\'ar, for every $0<c<c_V$ and every compact set of normalized
nonzero sections,
\begin{equation}\label{eq:uniform-cse-integral}
 \sup_s\int_M\|s\|^{-2c}\,d\lambda<\infty.
\end{equation}
Local $L^1$ stability below the complex singularity exponent is uniform near each
$([s],x)$; compactness of $\PP(V)\times M$ reduces the argument to finitely many
such neighborhoods; see \cite{DK01}.

\begin{proposition}\label{prop:logholder}
The function
\begin{equation}\label{eq:Psi}
 \Psi(s)=\int_M\log\|s(x)\|\,d\lambda(x),\qquad s\ne0,
\end{equation}
is finite and locally H\"older continuous.  More precisely, on every compact
coefficient set away from zero it is $C^{0,\alpha}$ for every
\begin{equation}\label{eq:alpha-cv}
 0<\alpha\le1,\qquad \alpha<2c_V.
\end{equation}
In particular, it is locally Lipschitz whenever $c_V>1/2$.  If
$d\lambda(M)=1$, it satisfies $\Psi(cs)=\log|c|+\Psi(s)$.  For holomorphic
coefficient maps $s(a)\ne0$, $\Psi(s(a))$ is plurisubharmonic.
\end{proposition}
\begin{proof}
Fix a compact coefficient neighborhood disjoint from the origin.  By
\cref{lem:negative},
\[
 \lambda\{x:\|s(x)\|<r\}\le Cr^\eta
\]
uniformly on that neighborhood.  The lower logarithmic tail satisfies
\begin{equation}\label{eq:tail}
 \int_{\|s\|<a}\log\frac a{\|s\|}\,d\lambda
 =\int_0^a\lambda\{\|s\|<r\}\,\frac{dr}{r}\le C'a^\eta.
\end{equation}
The tail estimate shows that $\Psi$ is finite and that the logarithmic tail is
locally uniform.

For $r,q>0$ and $0<\alpha\le1$,
\begin{equation}\label{eq:log-difference}
 |\log r-\log q|
 \le \frac1\alpha |r-q|^\alpha\bigl(r^{-\alpha}+q^{-\alpha}\bigr).
\end{equation}
For $r\ge q$, this follows from
$\log(1+u)\le u^\alpha/\alpha$ with $u=(r-q)/q$; the case $q\ge r$ is symmetric.
Finite dimensionality implies
\[
 \sup_M\bigl|\|s\|-\|t\|\bigr|\le C\|s-t\|.
\]
Let $\alpha$ satisfy \eqref{eq:alpha-cv}.  Choose
$c<c_V$ with $\alpha<2c$.  After normalizing the coefficient vectors (which changes the estimates only by
bounded factors on the fixed compact coefficient set), the uniform integrability
statement \eqref{eq:uniform-cse-integral}, together with
$r^{-\alpha}\le 1+r^{-2c}$, implies that
\[
 \sup_s\int_M\|s\|^{-\alpha}\,d\lambda<\infty.
\]
Because the zero set of a nonzero holomorphic section has $\lambda$-measure zero,
we may integrate \eqref{eq:log-difference} to obtain
\[
 |\Psi(s)-\Psi(t)|\le C_\alpha\|s-t\|^\alpha.
\]
When $c_V>1/2$, the same argument applies with $\alpha=1$ after choosing
$1/2<c<c_V$.

For fixed $x$, the map $a\mapsto\log\|s(a)(x)\|$ is plurisubharmonic, since
the metric factor is independent of $a$, and the same holds for its lower
truncations.  By \eqref{eq:tail}, the averages converge locally uniformly.  Their limit
$\Psi(s(a))$ is therefore finite and plurisubharmonic; the scaling relation is
immediate from the definition.  Repeating the argument on the fixed compact
coefficient set makes the modulus uniform across the finite-dimensional family.
\end{proof}

\subsection{Chow current and incidence formula}
Let
\[
 G=\Gr(n-m,\C^{n+1}),\qquad \ell=\dim_\C G,
\]
so that a point of $G$ represents a projective $(n-m-1)$-plane.  Let $d\nu_G$ be
the unique $U(n+1)$-invariant probability volume form on $G$, and equip the
Pl\"ucker line bundle $\mathcal O_G(1)$ with any smooth invariant Hermitian metric.  A
cycle $C$ of dimension $m$ and degree $D$ has a Chow form
\[
 R_C\in V_{\rm Ch}:=H^0(G,\mathcal O_G(D)),
\]
well defined up to scalar; its zero divisor consists of the planes meeting $C$,
with the cycle multiplicities.  The assignment $C\mapsto[R_C]$ defines the classical Chow form embedding
$\chi:\Chow_{m,D}(\PP^n)\hookrightarrow\PP(V_{\rm Ch})$; see \cite{GKZ} and the recent
review \cite{PS25}.

Apply \cref{prop:logholder} with $M=G$.  On a projective coefficient chart, choose
a holomorphic representative $s(a)$ and set
\begin{equation}\label{eq:Chowpotential}
 \Psi_{\rm Ch}(a)=\int_G\log\|s(a)(L)\|_h\,d\nu_G(L).
\end{equation}
On overlaps, two such potentials differ by $\log|g(a)|$ for a nowhere-vanishing
holomorphic function $g$, so their difference is pluriharmonic.  The local potentials
patch to a global positive closed current $\cT$ on $\PP(V_{\rm Ch})$,
with locally H\"older potentials.
Logarithmic Chow norms and the corresponding incidence formulas are classical;
see \cite[Section~4]{PS02}.  To fix the normalization, consider the incidence
correspondence
\[
 I=\{(x,L)\in\PP^n\times G:x\in L\},\qquad
 a:I\to\PP^n,\quad b:I\to G.
\]

\begin{lemma}\label{lem:crofton}
With $\int_{\PP^1}\omega=1$ and $\int_G d\nu_G=1$,
\begin{equation}\label{eq:crofton}
 a_*b^*(d\nu_G)=\omega^{m+1}.
\end{equation}
\end{lemma}
\begin{proof}
By $U(n+1)$-invariance, the left-hand side must equal
$c\,\omega^{m+1}$ for some constant $c$.  Pairing with $\omega^{n-m-1}$ and using
Fubini on the incidence variety reduces the expression to
\[
 c\int_{\PP^n}\omega^n
 =\int_I b^*(d\nu_G)\wedge a^*\omega^{n-m-1}
 =\int_G\left(\int_L\omega^{n-m-1}\right)d\nu_G(L)=1.
\]
Both projective integrals equal one under our normalization, so $c=1$.
\end{proof}

\begin{proposition}\label{prop:curvature}
Let $B\subset X_{\rm reg}$ be admissible.  On a local resolved presentation as in
\cref{def:admissible}, set
\[
  dV_{i,b}:=\frac{(e_{i,b}^*\omega)^m}{m!}
\]
and define
\begin{equation}\label{eq:WP}
  \Omega_B
  :=\frac1{DV_m}\sum_i a_i(p_i)_*
       \frac{e_i^*\omega^{m+1}}{(m+1)!}.
\end{equation}
The forms defined in \eqref{eq:WP} are independent of the resolved presentation and
glue to a smooth K\"ahler form on $B$, with
\begin{equation}\label{eq:curvature}
  \chi^*\cT=\sum_i a_i(p_i)_*e_i^*\omega^{m+1},
  \qquad
  \Omega_B=\frac1{D(m+1)}\chi^*\cT.
\end{equation}
If $g_B(v,w):=\Omega_B(v,Jw)$, then for every real $v\in T_bB$,
\begin{equation}\label{eq:g}
  g_B(v,v)=\frac1{DV_m}\sum_i a_i
  \int_{\widetilde Z_{i,b}}|N_{i,v}|^2\,dV_{i,b},
\end{equation}
where $N_{i,v}$ is the normal component of the resolved infinitesimal motion on
the locus of full rank, extended by zero on the null set where the rank drops.  For every
absolutely continuous curve $b:[0,T]\to B$,
\begin{equation}\label{eq:speed}
  |\dot\mu_{b(t)}|_{W_2}^2
  =g_B(\dot b(t),\dot b(t))
  \quad\text{for a.e. }t.
\end{equation}
For the Riemannian length distance $d_B$ of $g_B$,
\begin{equation}\label{eq:lower}
  W_2(\mu_b,\mu_c)\le d_B(b,c).
\end{equation}
\end{proposition}

\begin{proof}
Work on one analytic neighborhood carrying a resolved presentation.  Let
$p_B,p_G$ denote the projections from $B\times G$, and let
$\mathcal D\subset B\times G$ be the pullback of the universal Chow divisor.
For a local Chow form, Poincar\'e--Lelong takes the form
\begin{equation}\label{eq:PLpush}
  dd^c_B\Psi_{\rm Ch}
  =(p_B)_*\bigl([\mathcal D]\wedge p_G^*d\nu_G\bigr);
\end{equation}
the curvature term of the Hermitian metric on $\mathcal O_G(D)$ has too high a vertical
bidegree after wedging with the top-degree form $d\nu_G$.  For each resolved
component, form the incidence pullback
\[
  I_i=\widetilde Z_i\times_{\PP^n}I
\]
and let
\[
  r_i:I_i\longrightarrow B\times G,
  \qquad
  r_i\bigl(u,(e_i(u),L)\bigr)=(p_i(u),L)
\]
be the induced proper map.  Over the generic incidence locus the birational
isomorphism locus of $e_{i,b}$ shows that $r_i$ has degree one onto the
component Chow divisor; the coefficient $a_i$ records precisely the cycle
multiplicity.  Functoriality of proper pushforward of relative cycles identifies
\[
  [\mathcal D]=\sum_i a_i(r_i)_*[I_i]
\]
as cycles on $B\times G$.  Fubini and \cref{lem:crofton} imply
\[
  \chi^*\cT
  =\sum_i a_i(p_i)_*e_i^*\bigl(a_*b^*d\nu_G\bigr)
  =\sum_i a_i(p_i)_*e_i^*\omega^{m+1}.
\]
Since $V_m=1/m!$,
\[
  \frac1{DV_m(m+1)!}=\frac1{D(m+1)},
\]
which proves \eqref{eq:curvature}.  This calculation also shows that the local
forms \eqref{eq:WP} glue independently of the chosen presentation.

On the locus of full rank of a resolved fiber, split the infinitesimal motion into
tangential and normal parts.  Complex tangent and normal subspaces are
$\omega$-orthogonal, and evaluation of $e_i^*\omega^{m+1}/(m+1)!$ on
$v,Jv$ and vertical vectors equals
$|N_{i,v}|^2(e_{i,b}^*\omega)^m/m!$.  The complement is a proper analytic set
and has zero $dV_{i,b}$-measure, proving \eqref{eq:g}.  A fiberwise Cartan
calculation also gives, for every $\phi\in C^\infty(\PP^n)$,
\begin{equation}\label{eq:firstvariation-B}
  d\!\left(\int\phi\,d\mu_b\right)(v)
  =\frac1{DV_m}\sum_i a_i
    \int_{\widetilde Z_{i,b}}d\phi(N_{i,v})\,dV_{i,b}.
\end{equation}
The tangential terms cancel.  Remove small neighborhoods of the exceptional set,
the set where the rank drops, and the component intersections.  Tubular neighborhoods of
the remaining compact regular images then approximate $N_v$ in $L^2(\mu_b)$ by
ambient gradients.  The approximation places $N_v$ in the Wasserstein tangent
space.  Equation \eqref{eq:firstvariation-B} identifies it as the unique minimal-norm
velocity representing this infinitesimal variation.

Suppose \eqref{eq:g} vanishes.  Then every component normal motion vanishes, so
at a general incidence pair the corresponding component Chow divisor has zero
normal first-order displacement.
If $R_i$ is a local component Chow form, its derivative is divisible by
$R_i$; since both have the same degree on the compact Grassmannian, the quotient
is constant.  Each component Chow point therefore has zero derivative, and so does
the weighted total Chow point.  The inclusion $B\hookrightarrow X\hookrightarrow
\Chow_{m,D}(\PP^n)$ is an immersion, which forces $v=0$.  Smoothness follows directly
from fiber integration in \eqref{eq:WP}.

Let $b(t)$ be absolutely continuous and put
$v(t)=\sqrt{g_B(\dot b(t),\dot b(t))}$.  The fields
$N_t=N_{\dot b(t)}$ are Borel, satisfy the continuity equation by
\eqref{eq:firstvariation-B}, and have $L^2(\mu_t)$ norm $v(t)\in L^1(0,T)$.
Set
\[
  \beta(t)=\int_0^t(1+v(r))\,dr,
  \qquad \tau=\beta^{-1}.
\]
For $\widetilde\mu_s=\mu_{\tau(s)}$, the rescaled minimal velocity has norm
$v(\tau(s))/(1+v(\tau(s)))$, and
\[
 \int_0^{\beta(T)}
 \left(\frac{v(\tau(s))}{1+v(\tau(s))}\right)^2ds
 =\int_0^T\frac{v(t)^2}{1+v(t)}\,dt
 \le\int_0^T v(t)\,dt<\infty.
\]
The $AC^2$ minimal-velocity characterization in Wasserstein space identifies the
metric speed of $\widetilde\mu$ with this norm.  Applying the metric chain rule
to $s=\beta(t)$ returns
$|\dot\mu_t|_{W_2}=v(t)$ almost everywhere, which is \eqref{eq:speed}.
Integration along curves then gives \eqref{eq:lower}.  We use the formulation for
compact Riemannian manifolds of the continuity-equation and minimal-velocity
characterizations in \cite[Theorem~2.4 and Proposition~2.6]{ACL}, together with
the metric framework of \cite[Chapter~8]{AGS}.
\end{proof}

At the Chow boundary, compatible representatives satisfy
$R_{\sum a_iC_i}=\prod_iR_{C_i}^{a_i}$.  The potential \eqref{eq:Chowpotential}
therefore records the cycle multiplicities while retaining a uniform H\"older
modulus.

\begin{corollary}\label{cor:smooth-cycle-specialization}
Suppose, in addition, that every cycle parametrized by an admissible locus $B$
is smooth, reduced, and embedded.  Then the tautological relative cycle has a
canonical smooth projective subscheme representative
\[
  \pi:\cZ\longrightarrow B,
  \qquad \cZ\subset B\times\PP^n,
\]
and \eqref{eq:g} reduces to
\[
  g_{B,b}(v,w)=\frac1{DV_m}
  \int_{Z_b}g_{\mathrm{FS}}(N_v,N_w)\,d\Vol_{Z_b}.
\]
For smooth embedded cycles this is precisely the usual Wasserstein metric.
\end{corollary}
\begin{proof}
The base is reduced and the relative cycle is multiplicity-free with geometrically
normal fibers.  Rydh's normality criterion and Hilbert--Chow comparison
\cite[Part~IV, Theorem~12.8 and Corollary~12.9]{Rydh08} identify it with a unique normal
flat family of subschemes.  The fiberwise smoothness criterion
\cite[\href{https://stacks.math.columbia.edu/tag/01V8}{Tag~01V8}]{Stacks} makes the family smooth, and the closed
immersion into $B\times\PP^n$ makes it projective.  On smaller analytic
neighborhoods, the connected components of the smooth fibers may be labelled and
used as the embedded components of a resolved presentation.  Reducedness of the
fibers makes all multiplicities equal to one; summing \eqref{eq:g} over the
disjoint smooth components recovers the displayed integral over $Z_b$.
\end{proof}

Denote by $d_B$ the Riemannian length distance of $g_B$, and write
$\widehat B:=\widehat{(B,d_B)}$ for its metric completion.

\begin{theorem}[Metric normalization]\label{thm:main}
For every admissible dense smooth Zariski-open locus $B\subset X_{\rm reg}$, the
identity on $B$ extends to a canonical homeomorphism
\begin{equation}\label{eq:main}
  \Phi:Y\longrightarrow\widehat B.
\end{equation}
Under this identification, set
\[
  \bar d(y,z)=d_{\widehat B}(\Phi(y),\Phi(z)).
\]
Then $(Y,\bar d)$ is compact, complete, and geodesic, and
$\bar d|_{B\times B}=d_B$.  For any sequences $b_j,c_j\in B$ converging to
$y,z\in Y$ in the normalization topology,
\begin{equation}\label{eq:limitmetric}
  \bar d(y,z)=\lim_{j\to\infty}d_B(b_j,c_j).
\end{equation}
For every chordal distance $h_Y$ induced by a projective embedding of $Y$, there
exist $C,\beta>0$ such that
\begin{equation}\label{eq:holdermetric}
  \bar d(y,z)\le C h_Y(y,z)^\beta,
\end{equation}
which implies
\begin{equation}\label{eq:LC}
  \diam_{d_B}\bigl(B\cap B_{h_Y}(y,r)\bigr)
  \le 2Cr^\beta\longrightarrow0
\end{equation}
uniformly in $y$.  The resulting metric on $Y$ is independent of the admissible
locus; see \cref{cor:admissible-locus-independent}.
\end{theorem}

\section{Metric normalization}\label{sec:normalization}

This section identifies the metric completion with the normalization.  Local
intrinsic-distance control and weak-gradient rigidity first collapse resolution
fibers, while the ambient Wasserstein lower bound separates distinct branches.
A semialgebraic comparison then yields the global H\"older estimate
\eqref{eq:holdermetric}.  Admissible-locus independence follows later from the
global action identity; see \cref{cor:admissible-locus-independent}.

\subsection{Local intrinsic distance control}
The key analytic input is a local estimate that converts H\"older control of the
potential into intrinsic distance control; compare
\cite[Proposition 1.4 and Section 4.2]{GGZ}.

\begin{lemma}\label{lem:localdistance}
Let $U$ be a Euclidean ball in $\C^N$, $A\subset U$ a proper analytic subset, and
$T=dd^cu\ge0$ with $u\in C^{0,\alpha}(U)$, $\alpha>0$.  Suppose $T$ is a smooth
K\"ahler form on $U\setminus A$, with associated metric $g$.  For sufficiently
nearby $x,y$ in a smaller ball outside $A$, there exists a piecewise smooth curve
in $U\setminus A$ joining them with
\begin{equation}\label{eq:localdistance}
 L_g(\gamma)\le C|x-y|^{\alpha/2}.
\end{equation}
The constant is uniform in $x,y$ in that smaller ball.
\end{lemma}
\begin{proof}
Set $d=2N$ and $e=\tr_{\rm Eucl}g$ off $A$.  Extend $e$ by zero on $A$ for
Lebesgue integration.  A cutoff $\chi$ equal to one on $B(x,r)$ and supported in
$B(x,2r)$ has $|dd^c\chi|\le Cr^{-2}$.  After subtracting $u(x)$, integration by parts implies the trace estimate
\begin{equation}\label{eq:mass}
 \int_{B(x,r)}e\,dz\le C r^{d-2+\alpha}.
\end{equation}
Because the estimate is obtained at the level of the positive current before
restriction to the smooth locus, it also holds for balls centered on $A$.

On each dyadic annulus, Cauchy--Schwarz bounds the radial integral by
\begin{align*}
 &\int_{S^{d-1}}\int_{r/2}^r\sqrt{e(x+t\theta)}\,dt\,d\theta\\
 &\quad\le
 \left(\int_{B(x,r)\setminus B(x,r/2)}e\,dz\right)^{1/2}
 \left(|S^{d-1}|\int_{r/2}^r t^{1-d}\,dt\right)^{1/2}
 \le Cr^{\alpha/2}.
\end{align*}
For $d=2$ the second integral is $\log2$; for $d>2$ it is of order $r^{2-d}$.
Summing the dyadic annuli proves
\begin{equation}\label{eq:radialmean}
 \int_{S^{d-1}}\int_0^r\sqrt{e(x+t\theta)}\,dt\,d\theta\le Cr^{\alpha/2}.
\end{equation}
If $r=|x-y|$, the intersection $B(x,2r)\cap B(y,2r)$ has volume at least
$c r^d$.  For $z=x+\rho\theta$ define
\[
 L_x(z):=\int_0^\rho\sqrt{e(x+t\theta)}\,dt,
\]
using the extension $e=0$ on $A$ fixed above.  This quantity is defined whether
or not the radial segment meets $A$; when the segment stays in $U\setminus A$,
its $g$-length is bounded above by $L_x(z)$, because
$g(\theta,\theta)\le\tr_{\rm Eucl}g=e$.  Integrating \eqref{eq:radialmean} once more over
$0<\rho<2r$, with weight $\rho^{d-1}$, produces the estimate
\[
   \int_{B(x,2r)}L_x(z)\,dz\le C r^{d+\alpha/2},
\]
and the analogous estimate holds with $y$ in place of $x$.  Restricting to
$E:=B(x,2r)\cap B(y,2r)$ and dividing by $|E|\ge c r^d$ shows that
\[
 \frac1{|E|}\int_E\bigl(L_x(z)+L_y(z)\bigr)\,dz
 \le C_0r^{\alpha/2}.
\]
Chebyshev's inequality shows that the set
\[
 E_{\rm good}:=\{z\in E:L_x(z)+L_y(z)\le 2C_0r^{\alpha/2}\}
\]
has measure at least $|E|/2$.  For fixed $x\notin A$, the endpoints for which the
radial segment meets $A$ lie locally in
\[
 \mathcal B_x=\{x+\lambda(a-x):a\in A,\ \lambda\ge1\}.
\]
A proper complex analytic subset has real dimension at most $d-2$.  The locally
relevant part of this image therefore has real dimension at most $d-1$ and zero
Lebesgue measure.  The analogous exceptional set for $y$ is also null, so
$E_{\rm good}\setminus(\mathcal B_x\cup\mathcal B_y)$ is nonempty.  Choosing $z$
there, the broken radial path from $x$ to $z$ to $y$ lies in $U\setminus A$, has
$g$-length at most $2C_0r^{\alpha/2}$, and proves the assertion.
\end{proof}

\subsection{Rigidity on constant-potential fibers}\label{sec:null}
To determine which points of a resolution are identified in the completion, we
need two weak-gradient facts.

For a real-valued function $F\in W^{1,2}_{\rm loc}$, write
\begin{equation}\label{eq:Qdef}
 Q(F):=2i\,\partial F\wedge\bar\partial F.
\end{equation}
With the convention $\Omega(v,Jv)=g(v,v)$, a direct calculation in a unitary
holomorphic frame shows
\begin{equation}\label{eq:Qunit}
 Q(F)\le \Omega\quad\Longleftrightarrow\quad |dF|_g\le1
\end{equation}
for smooth $F$; the same equivalence holds almost everywhere for Sobolev functions.

\begin{lemma}\label{lem:weakgradient}
Let $R$ be a smooth connected compact complex manifold, $A$ an SNC divisor, and let
$T$ be a positive closed $(1,1)$-current with locally H\"older potentials.  Assume
that $T$ is a smooth K\"ahler form on $R\setminus A$, with length distance $d$.
For $p\in R\setminus A$, the distance function
$F_p=d(p,\cdot)$ extends continuously to $R$, belongs to
$W^{1,2}_{\rm loc}(R)$, and satisfies the current inequality with constant one
\begin{equation}\label{eq:gradcurrent}
 Q(F_p)\le T.
\end{equation}
\end{lemma}
\begin{proof}
If two points of $R\setminus A$ approach the same point of $R$,
\cref{lem:localdistance} forces their mutual $d$-distance to vanish.  The inclusion of $R\setminus A$ into its metric completion extends continuously
across $A$, locally with a H\"older modulus.  The distance function $F_p$ therefore
extends continuously to $R$ and is locally bounded there.

Fix a coordinate ball $U\Subset R$ with a smooth background K\"ahler form
$\omega_0$.  On $U\setminus A$, the distance function is locally Lipschitz for the
smooth metric $T$ and satisfies $|dF_p|_T\le1$ almost everywhere.  By
\eqref{eq:Qunit},
\[
 Q(F_p)\le T\qquad\text{a.e.\ on }U\setminus A.
\]
Taking the $\omega_0$-trace and integrating, we obtain
\[
 \int_{U\setminus A}|\nabla_{\omega_0}F_p|^2\,dV_{\omega_0}
 \le C\int_U T\wedge\omega_0^{N-1}<\infty.
\]
The classical gradient is square-integrable off $A$.  An SNC divisor has zero
Lebesgue measure, so extending the gradient by zero across $A$ defines an
$L^2(U)$ one-form, still denoted by $dF_p$.

A logarithmic cutoff removes the possible Sobolev boundary term along the divisor.
In SNC coordinates choose $\chi_\eps$ which vanish when a normal
coordinate has modulus at most $\eps^2$, equal one once all such moduli are at least
$\eps$, and satisfy
\[
 \int_U|d\chi_\eps|_{\omega_0}^2\,dV_{\omega_0}
 =O(|\log\eps|^{-1})\longrightarrow0.
\]
Products of the one-variable cutoffs handle finitely many crossing components.
Local boundedness of $F_p$ gives
$F_p\chi_\eps\to F_p$ in $L^2$ and
$F_p\,d\chi_\eps\to0$ in $L^2$.  Moreover,
$|(1-\chi_\eps)dF_p|^2\le |dF_p|^2\in L^1(U)$ and
$\chi_\eps\to1$ almost everywhere off $A$, so dominated convergence shows that
$(1-\chi_\eps)dF_p\to0$ in $L^2$.  These cutoff estimates place $F_p$ in
$W^{1,2}(U)$, with weak derivative equal to its classical derivative on
$U\setminus A$.

The weak gradient is square-integrable, so the matrix-valued measure $Q(F_p)$ is
absolutely continuous with respect to Lebesgue measure and assigns no mass to $A$.
Testing the pointwise
inequality on $U\setminus A$ against arbitrary nonnegative compactly supported
$(N-1,N-1)$ forms and passing across $A$ proves \eqref{eq:gradcurrent} with the
same constant one.
\end{proof}

\begin{lemma}\label{lem:null}
Let $M$ be a complex manifold, let $u\in C^0(M)\cap\mathrm{PSH}(M)$, and let
$F\in C^0(M)\cap W^{1,2}_{\rm loc}(M)$ be real-valued.  Suppose that in every
coordinate domain
\begin{equation}\label{eq:null-hyp}
 i\partial F\wedge\bar\partial F\le C\,dd^cu
\end{equation}
as positive $(1,1)$-currents, with a locally uniform constant $C$.  If
$h:\Delta\to M$ is a holomorphic disk and $u\circ h$ is constant, then
$F\circ h$ is constant.  The conclusion extends to every connected compact
analytic subset $Z\subset M$ on which $u|_Z$ is locally constant.
\end{lemma}
\begin{proof}
First work locally on the disk.  Around a point of $\Delta$, choose
$\Delta_0\Subset\Delta$ whose image lies in a coordinate domain $U\Subset M$, and
regularize in the ambient coordinates with a standard nonnegative real mollifier
$\rho_\eps$:
\[
 F_\eps=F*\rho_\eps,\qquad u_\eps=u*\rho_\eps.
\]
On compact subsets, $F_\eps\to F$ uniformly and in $W^{1,2}_{\rm loc}$, while
$u_\eps\to u$ uniformly and $dd^cu_\eps=(dd^cu)*\rho_\eps$.  Interpreting
$i\partial F\wedge\bar\partial F$ as the positive Hermitian matrix-valued
$L^1$ measure associated with the weak complex gradient, the matrix
Cauchy--Schwarz/Jensen inequality implies
\begin{equation}\label{eq:jensen}
 i\partial F_\eps\wedge\bar\partial F_\eps
 \le (i\partial F\wedge\bar\partial F)*\rho_\eps
 \le C\,dd^cu_\eps.
\end{equation}
After regularization both sides are smooth, so pulling back by $h$ is legitimate.

Let $\zeta\ge0$ be smooth and compactly supported in $\Delta_0$, equal to one on a
smaller concentric disk.  From \eqref{eq:jensen},
\[
 \int_{\Delta_0} \zeta\,
 i\partial(F_\eps\circ h)\wedge\bar\partial(F_\eps\circ h)
 \le C\int_{\Delta_0}\zeta\,dd^c(u_\eps\circ h).
\]
Integration by parts rewrites the right-hand side as
\[
 \int_{\Delta_0}\zeta\,dd^c(u_\eps\circ h)
 =\int_{\Delta_0}(u_\eps\circ h)\,dd^c\zeta\longrightarrow0,
\]
because $u_\eps\circ h$ converges uniformly to a constant and
$\int_{\Delta_0} dd^c\zeta=0$.  The Dirichlet energy of $F_\eps\circ h$ tends to zero on every smaller disk.
Poincar\'e's inequality and uniform convergence then force $F\circ h$ to be
constant there.  The constants agree on overlaps; connectedness of $\Delta$ then makes $F\circ h$
constant on the whole disk.

For a connected compact analytic set $Z\subset M$, the regular locus of each
positive-dimensional irreducible component is connected and locally covered by
holomorphic disks, so the disk argument makes $F$ constant on that regular locus
and continuity extends the constancy to the entire component.  There are only
finitely many irreducible components.  Since $Z$ is connected, their intersection graph is
connected as well, and the constants agree from component to component.  The
zero-dimensional case is immediate.
\end{proof}

\subsection{Fiber collapse and branch separation}
We construct the completion map on a resolution, show that it is constant on
resolution fibers, and then separate distinct normalization branches.  This last
point is essential: branches above the same Chow cycle have
the same limiting measure, so their separation must come from the positive cost
of paths that leave a fixed neighborhood of the full normalization fiber.

Take a projective resolution
\[
 \rho:R\longrightarrow Y
\]
which is an isomorphism over $B$ and makes $A=R\setminus B$ SNC\@.  Set
\[
 f=\chi\circ\nu\circ\rho:R\longrightarrow\PP(V_{\rm Ch}),\qquad
 T_R=\frac1{D(m+1)}f^*\cT.
\]
The local H\"older potentials from \cref{prop:logholder} define this pullback.
Over $B$, \cref{prop:curvature} identifies $T_R$ with $\Omega_B$.
Local holomorphic coefficient representatives composed with $f$ remain Lipschitz
on compact coordinate sets, so these potentials remain H\"older on $R$.

Fix a chordal distance $h_R$ coming from a projective embedding of $R$.  We first
extend the completion map across the resolution.  For $r\in R$, choose any sequence $b_j\in B$
converging to $r$ in the analytic topology.  In a coordinate neighborhood of $r$,
\cref{lem:localdistance}, together with equivalence of the local Euclidean and
projective distances, implies
\[
 d_B(b_j,b_k)\le C_r h_R(b_j,b_k)^{\theta_r}\longrightarrow0
\]
for some $C_r<\infty$ and $\theta_r>0$, so $(b_j)$ is $d_B$-Cauchy.
Interlacing two sequences approaching the same $r$ shows that their completion
limits agree.  Set
\begin{equation}\label{eq:K}
 K(r):=\lim_{j\to\infty}b_j\in\widehat B.
\end{equation}
The preceding estimate also gives continuity of $K$; on $B$ it is the identity.
Fix any exponent $\alpha$ allowed by \cref{prop:logholder} for the compact Chow
family.  The coefficient maps on a finite coordinate cover of $R$ are uniformly
Lipschitz after enlarging the constants, so the same exponent $\alpha$ applies on
every chart.  After enlarging the constant, the finite cover establishes
\begin{equation}\label{eq:Kholder}
 d_{\widehat B}(K(r),K(s))\le C h_R(r,s)^{\alpha/2}.
\end{equation}
For pairs outside the corresponding local neighborhoods, compactness absorbs the
estimate, so the exponent $\alpha/2$ holds globally on $R$.  To prove surjectivity,
let $\xi\in\widehat B$ be represented by a $d_B$-Cauchy sequence $(b_j)\subset B$.  Compactness of $R$ allows us to extract a subsequence
converging analytically to some $r\in R$.  The construction shows that this subsequence converges in $\widehat B$ to $K(r)$.  Uniqueness of the
Cauchy limit forces $\xi=K(r)$, so $K:R\to\widehat B$ is onto.

Fix $y\in Y$ and let $E_y=\rho^{-1}(y)$.  The fiber $E_y$ is connected because
$Y$ is normal and $\rho$ is proper birational,
by Stein factorization \cite[\href{https://stacks.math.columbia.edu/tag/03H0}{Tag~03H0}]{Stacks}.  Choose a coefficient chart
at $\chi\nu(y)$ and its normalized holomorphic section representative.  The potential
\[
 u=\frac1{D(m+1)}\Psi_{\rm Ch}\circ f
\]
is constant on $E_y$. For every $p\in B$, \cref{lem:weakgradient,lem:null} imply that
$r\mapsto d_{\widehat B}(p,K(r))$ is constant on $E_y$.  Since distance functions
to the dense subset $B$ separate points of $\widehat B$,
\begin{equation}\label{eq:fibercollapse}
 K(r)=K(s)\quad\text{whenever }\rho(r)=\rho(s).
\end{equation}
Properness makes $\rho$ a quotient map, so $K$ factors through a continuous
surjection $\Phi:Y\to\widehat B$.

To prove injectivity, first suppose $\nu(y_1)\ne\nu(y_2)$.  The lower bound
\eqref{eq:lower} and injectivity of $C\mapsto\mu_C$ separate their completion
points.  Suppose instead that $y_1\ne y_2$ lie above the same $x\in X$.
The finite normalization fiber is $\nu^{-1}(x)=\{y_1,\ldots,y_a\}$.  Since
$\dim Y>0$, choose pairwise disjoint neighborhoods $U_i$ of the $y_i$ whose union
does not cover $Y$, and set $K_0=Y\setminus\bigcup_iU_i$.  Then $K_0$ is
nonempty and compact, while $\nu(K_0)$ avoids $x$.  Compactness and separation ensure
\[
 \delta=\inf_{z\in\nu(K_0)}W_2(\mu_z,\mu_x)>0.
\]
Both endpoint measures converge to $\mu_x$; separation therefore comes from the
positive cost of any connecting path that leaves a fixed neighborhood of the full normalization fiber.  Let $b_j\in B\cap U_1$ and $c_j\in B\cap U_2$ converge to $y_1$ and $y_2$,
respectively, and let $\Gamma_j\subset B$ be any piecewise smooth curve joining
$b_j$ to $c_j$; write $L_{g_B}(\Gamma_j)$ for its $g_B$-length.  Because the $U_i$ are pairwise disjoint, $\Gamma_j$ must leave their union; choose
$z_j\in\Gamma_j\cap K_0$.  Using \eqref{eq:lower} and the triangle
inequality in $W_2$,
\[
  \begin{aligned}
  L_{g_B}(\Gamma_j)
  &\ge d_B(b_j,z_j)
   \ge W_2(\mu_{b_j},\mu_{\nu(z_j)})\\
  &\ge W_2(\mu_x,\mu_{\nu(z_j)})-W_2(\mu_{b_j},\mu_x)
   \ge \delta-W_2(\mu_{b_j},\mu_x).
  \end{aligned}
\]
The continuity of $C\mapsto\mu_C$ in the Chow topology implies
$W_2(\mu_{b_j},\mu_x)\to0$.  Taking the infimum over all $\Gamma_j$ and then the
lower limit, we obtain
\[
  \liminf_{j\to\infty}d_B(b_j,c_j)\ge\delta>0.
\]
The two branches are therefore separated by a positive distance, which proves that
$\Phi$ is injective.

The compactness of $Y$ and metrizability of $\widehat B$ make $\Phi$ a
homeomorphism.  It is the unique continuous extension of the identity on $B$,
which also makes the construction independent of the chosen resolution.  The
completion of a length space is again a length space, and compactness then implies
geodesicity.

\subsection{H\"older control on the normalization}
To compare $\bar d$ with the projective geometry of $Y$, choose chordal distances
coming from projective embeddings of $R$ and $Y$.  On the compact real algebraic set $R\times R$, define
\[
 Z_\rho=\{(r,s):\rho(r)=\rho(s)\}.
\]
A compact semialgebraic \L{}ojasiewicz inequality gives constants
$C,\eta>0$ such that
\begin{equation}\label{eq:loja}
 \dist_{R\times R}((r,s),Z_\rho)
 \le C h_Y(\rho(r),\rho(s))^\eta.
\end{equation}
Indeed, the two nonnegative continuous semialgebraic functions
\[
 a(r,s):=\dist_{R\times R}((r,s),Z_\rho),
 \qquad
 b(r,s):=h_Y(\rho(r),\rho(s))
\]
on the compact set $R\times R$ have the same zero set $Z_\rho$; the usual
compact \L{}ojasiewicz comparison therefore yields $a^N\le Cb$ for some
integer $N\ge1$, which is equivalent to \eqref{eq:loja}.  Compare also the
related distance comparison in \cite[Proposition~4.6]{GGZ}.

Choose a nearest $(r',s')\in Z_\rho$.  As $K(r')=K(s')$,
\eqref{eq:Kholder} gives
\begin{align*}
 \bar d(\rho(r),\rho(s))
 &\le d_{\widehat B}(K(r),K(r'))+d_{\widehat B}(K(s'),K(s))\\
 &\le C\dist_{R\times R}((r,s),Z_\rho)^{\alpha/2}\\
 &\le C' h_Y(\rho(r),\rho(s))^{\eta\alpha/2}.
\end{align*}
Surjectivity of $\rho$ yields \eqref{eq:holdermetric} with
$\beta=\eta\alpha/2>0$.  The diameter estimate \eqref{eq:LC} follows by connecting
both endpoints through $y$ in the completed metric and using
$\bar d|_{B\times B}=d_B$.  This proves the metric normalization and H\"older
assertions of \cref{thm:main}, and hence \cref{thm:intro-main}.

\begin{corollary}
\label{cor:chow-inner-compactness}
For $x_0,x_1\in X$, define
\begin{equation}\label{eq:chow-inner-action}
  D_X(x_0,x_1)^2
  :=
  \inf_{\substack{\gamma\in C([0,1],X),\ \gamma(0)=x_0,\ \gamma(1)=x_1\\
                     t\mapsto\mu_{\gamma(t)}\in AC^2(W_2)}}
  \int_0^1
  \left|\frac d{dt}\mu_{\gamma(t)}\right|_{W_2}^2\,dt.
\end{equation}
Then $D_X$ is a finite metric on $X$, induces the usual Chow topology, and makes
$X$ compact.  For every $y_0,y_1\in Y=X^\nu$,
\begin{equation}\label{eq:chow-inner-upper}
  D_X\bigl(\nu(y_0),\nu(y_1)\bigr)
  \le \bar d(y_0,y_1).
\end{equation}
When $X$ is nonnormal, $D_X$ forgets which normalization branch is followed,
whereas $\bar d$ retains that information.
\end{corollary}
\begin{proof}
The inequality $W_2(\mu_{\nu(y)},\mu_{\nu(z)})\le\bar d(y,z)$ follows from
\eqref{eq:lower} by passage to the metric completion.  Let $\eta$ be a
constant-speed $\bar d$-geodesic from $y_0$ to $y_1$.  Its measure image is
$\bar d(y_0,y_1)$-Lipschitz in $W_2$ and belongs to $AC^2$, while the
projected curve $\nu\circ\eta$ is admissible in \eqref{eq:chow-inner-action}.
Its action is at most $\bar d(y_0,y_1)^2$, proving
\eqref{eq:chow-inner-upper} and finiteness of $D_X$.

Symmetry follows by time reversal, and the triangle inequality by the usual
constant-speed reparametrization and concatenation after rescaling time.
Moreover,
\[
  W_2(\mu_x,\mu_{x'})\le D_X(x,x'),
\]
so injectivity of the cycle measure map at fixed degree shows that $D_X$ separates
points.  Equation \eqref{eq:chow-inner-upper} makes
$\nu:(Y,\bar d)\to(X,D_X)$ a continuous surjection, so compactness of
$(Y,\bar d)$ gives compactness of $(X,D_X)$.  Conversely, $D_X$-convergence
implies $W_2$-convergence of the cycle measures.  The continuous injection
$x\mapsto\mu_x$ from compact $X$ into the Hausdorff space
$\mathcal P_2(\PP^n)$ is a homeomorphism onto its image.
These two implications identify the $D_X$-topology with the original Chow
topology.
\end{proof}

\begin{corollary}\label{thm:effective-holder}
For the Chow family set
\begin{equation}\label{eq:cch}
 c_{\rm Ch}(X):=\inf_{([s],L)\in\chi(X)\times G}c_L(s)>0,
\end{equation}
where $s$ denotes any nonzero representative of the projective Chow form $[s]$;
the exponent $c_L(s)$ is unchanged by rescaling.  Let $\lambda_\rho>0$ satisfy
\begin{equation}\label{eq:lambda-rho}
 \dist_{R\times R}((r,s),Z_\rho)
 \le C h_Y(\rho(r),\rho(s))^{\lambda_\rho}.
\end{equation}
Then every
\begin{equation}\label{eq:beta-effective}
 0<\beta<\lambda_\rho\min\left\{c_{\rm Ch}(X),\frac12\right\}
\end{equation}
is admissible in \eqref{eq:holdermetric}, after changing the multiplicative
constant.  If $c_{\rm Ch}(X)>1/2$, the endpoint
$\beta=\lambda_\rho/2$ is admissible as well.
\end{corollary}
\begin{proof}
For every $0<c<c_{\rm Ch}(X)$, Demailly--Koll\'ar semicontinuity provides a
uniform negative moment of order $2c$ on the compact Chow family.  The direct
logarithmic estimate in \cref{prop:logholder} applies to every local exponent
\[
 0<\alpha\le1,\qquad \alpha<2c_{\rm Ch}(X),
\]
with $\alpha=1$ allowed when $c_{\rm Ch}(X)>1/2$.  Repeating the local-distance argument with this exponent and then passing to a
finite cover gives
\[
 d_{\widehat B}(K(r),K(s))\le C h_R(r,s)^{\alpha/2}.
\]
Together with \eqref{eq:lambda-rho}, this shows that every
$\beta<\lambda_\rho\min\{c_{\rm Ch}(X),1/2\}$ is admissible in
\eqref{eq:holdermetric}.  If
$c_{\rm Ch}(X)>1/2$, taking $\alpha=1$ also allows the endpoint
$\beta=\lambda_\rho/2$.
\end{proof}

\begin{corollary}\label{cor:smooth-normalization-holder}
Assume that $Y=X^\nu$ is smooth.  Then for every chordal distance $h_Y$ induced
by a projective embedding of $Y$ and every
\[
  0<\beta<\min\left\{c_{\rm Ch}(X),\frac12\right\}
\]
there is $C_\beta<\infty$ such that
\begin{equation}\label{eq:smooth-normalization-holder}
  \bar d(y,z)\le C_\beta h_Y(y,z)^\beta
  \qquad (y,z\in Y).
\end{equation}
If $c_{\rm Ch}(X)>1/2$, the endpoint $\beta=1/2$ is admissible.
\end{corollary}
\begin{proof}
Set $A:=Y\setminus B$.  Because $Y$ is smooth and $B$ is Zariski open, $A$ is a
proper analytic subset.  The current
\[
  T_Y:=\frac1{D(m+1)}(\chi\nu)^*\cT
\]
has locally H\"older potentials on all of $Y$ and, by
\cref{prop:curvature}, is the smooth K\"ahler form $\Omega_B$ on $B$.  For every
$0<\alpha\le1$ with $\alpha<2c_{\rm Ch}(X)$, the Chow potential is locally
$C^{0,\alpha}$; when $c_{\rm Ch}(X)>1/2$ one may take $\alpha=1$.  Applying \cref{lem:localdistance} in coordinate balls of $Y$ with singular set
$A$ and then passing to a finite cover, we obtain
\[
  \bar d(y,z)\le C h_Y(y,z)^{\alpha/2}
\]
directly on $Y$.  Taking $\beta=\alpha/2$ proves the stated range and endpoint.
\end{proof}

With the metric normalization and H\"older conclusions established, we record
two applications before returning in \cref{sec:singular-strata} to the general
cycle family.

\subsection{Inner Wasserstein geometry of hypersurfaces}

For hypersurfaces the Chow-form identification is
\[
  P_{n,d}:=\PP H^0(\PP^n,\mathcal O_{\PP^n}(d))
  \simeq \Chow_{n-1,d}(\PP^n).
\]
The parameter space $P_{n,d}$ is smooth and normal.  Over the smooth locus
$B=P_{n,d}\setminus\Delta_{n,d}$, the universal hypersurface
\[
  \mathcal Z_B=\{(p,z)\in B\times\PP^n:p(z)=0\}\longrightarrow B
\]
is smooth and proper because the derivative in the $\PP^n$ direction never
vanishes along a fiber.  Its connected components can be labelled locally, so
$B$ is admissible.

\begin{lemma}
\label{lem:hyp-critical-potential}
Fix a Hermitian norm on $H^0(\PP^n,\mathcal O(d))$, and let
$d\lambda=\omega_0^n$ be the probability Fubini--Study volume measure for the
normalization with $\int_{\PP^1}\omega_0=1$.  There are constants $C,r_0>0$ such that every
normalized nonzero degree-$d$ polynomial $p$ satisfies
\begin{equation}\label{eq:critical-sublevel}
  \lambda\bigl\{z:\|p(z)\|_{\mathcal O(d)}<r\bigr\}
  \le C r^{2/d},
  \qquad 0<r<r_0.
\end{equation}
Moreover,
\begin{equation}\label{eq:hypersurface-cch}
  c_{\rm Ch}(P_{n,d})=\frac1d.
\end{equation}
If $d\ge2$, the hypersurface Chow potential is locally $C^{0,2/d}$ on
$P_{n,d}$; in particular, it is locally Lipschitz when $d=2$.
\end{lemma}
\begin{proof}
Work first near a fixed normalized coefficient--point pair $(p_0,x_0)$.  If
$p_0(x_0)\ne0$, then after shrinking the coefficient--point neighborhood the
section norm is uniformly bounded away from zero, so
\eqref{eq:critical-sublevel} is trivial there for sufficiently small $r$.  We may
therefore assume $p_0(x_0)=0$.  Choose an affine chart centered at $x_0$ and a
complex affine-linear spatial change of coordinates so that the restriction of
$p_0$ to the $z_1$-axis is not identically zero.  Its vanishing order there is
some $1\le a\le d$.  The parameter-dependent Weierstrass preparation theorem,
as in \cref{lem:negative}, then gives a factorization uniform in the parameters,
\[
  p(z)=u(p,z)\prod_{j=1}^a(z_1-\zeta_j),
  \qquad 1\le a\le d,
\]
with $|u|\ge c_0>0$ and all roots in a fixed bounded set.  If $|p(z)|<r$, then
at least one factor satisfies
\[
  |z_1-\zeta_j|<(r/c_0)^{1/a}.
\]
For each fixed $z'$, the corresponding one-variable sublevel set is
contained in the union of at most $a$ disks and has area at most
$Cr^{2/a}\le Cr^{2/d}$ for $0<r<1$.  Integrating in $z'$ and covering the compact product of the normalized coefficient
sphere with $\PP^n$ by finitely many such charts proves
\eqref{eq:critical-sublevel}.

For every $c<1/d$, the same Weierstrass chart has
$a(2c)\le2dc<2$, so the negative moment argument of \cref{lem:negative} is
uniform and proves $c_{\rm Ch}(P_{n,d})\ge1/d$.  Conversely, for $p=L^d$ at a
regular point of the hyperplane $L=0$ one has
$|p|^{-2c}\asymp|z_1|^{-2dc}$, which is locally integrable precisely when
$c<1/d$.  The two inequalities give \eqref{eq:hypersurface-cch}.

Suppose first that $d\ge3$.  Put $\kappa=2/d<1$ and work on a fixed coefficient
chart, where
\[
  \Psi(p)=\int_{\PP^n}\log\|p(z)\|\,d\lambda(z).
\]
For $\varepsilon>0$ define the truncated potential
\[
  \Psi_\varepsilon(p)
  :=\int_{\PP^n}\log\max\{\|p(z)\|,\varepsilon\}\,d\lambda(z).
\]
The layer-cake formula and \eqref{eq:critical-sublevel} imply the uniform estimate
\begin{equation}\label{eq:critical-log-tail}
  0\le\Psi_\varepsilon(p)-\Psi(p)
  =\int_0^\varepsilon
    \lambda\{\|p\|<r\}\,\frac{dr}{r}
  \le C\varepsilon^\kappa.
\end{equation}
If $p,q$ lie in a fixed compact coefficient chart and
$\delta=\|p-q\|$, finite dimensionality implies
$\sup_z|\|p(z)\|-\|q(z)\||\le A\delta$.  Set
$\varepsilon=A\delta$ and assume $\delta$ is small.  On the set where either
norm is below $2\varepsilon$, the two truncated logarithms differ by a bounded
constant and \eqref{eq:critical-sublevel} bounds this contribution by
$O(\varepsilon^\kappa)$.  On the complementary set the mean-value theorem implies
\[
  |\log\|p\|-\log\|q\||
  \le C\varepsilon\bigl(\|p\|^{-1}+\|q\|^{-1}\bigr).
\]
For either $f=\|p\|$ or $f=\|q\|$, layer cake applied to $1/f$ on
$\{f\ge2\varepsilon\}$ and the sublevel estimate imply
\[
  \int_{\{f\ge2\varepsilon\}}\frac{d\lambda}{f}
  \le C\left(1+\int_{2\varepsilon}^{r_0}r^{\kappa-2}\,dr\right)
  \le C\varepsilon^{\kappa-1}.
\]
Combining this with \eqref{eq:critical-log-tail}, we find
\[
  |\Psi(p)-\Psi(q)|\le C\delta^{2/d}.
\]
Thus the potential is locally $C^{0,2/d}$ for $d\ge3$.

The endpoint $d=2$ requires a different argument because the critical negative
moment need not be finite.  Fix $p_0\ne0$ and a relatively compact coefficient
neighborhood of $p_0$.  Using the same parameter-dependent preparation as
above, compactness of $\PP^n$ gives finitely many spatial preparation charts,
together with nonvanishing charts, and a smooth partition of unity independent
of the coefficient $p$.  In one preparation chart write the spatial variables
as $(z,\zeta)\in\C\times\C^{n-1}$, with the $\zeta$ variable absent when $n=1$.
Absorbing the partition function and smooth volume density leaves a weighted
local integral with a fixed function $g(z,\zeta)\in C_c^\infty$.
The logarithm of the Weierstrass unit is smooth in all parameters and the
Hermitian frame factor is independent of $p$, so only the monic factor needs
attention.

For fixed $\zeta$ define
\[
  G_\zeta(a):=\int_{\C}g(z,\zeta)\log|z-a|\,dA(z).
\]
The functions $G_\zeta$ have uniform $C^2$ bounds for $a$ and $\zeta$ in the
relevant compact sets.  Indeed, $\log|\cdot|$ is locally integrable and, for
every real multi-index $|\alpha|\le2$, distributional integration by parts gives
\[
  D_a^\alpha G_\zeta(a)
  =\int_{\C}\log|z-a|\,D_z^\alpha g(z,\zeta)\,dA(z).
\]
The $z$-supports lie in one compact set and the derivatives $D_z^\alpha g$ are
uniformly bounded in $\zeta$; uniform local integrability of translates of
$\log|\cdot|$ therefore yields the required $C^2$ bounds.  A degree-one
Weierstrass factor $z-r_p(\zeta)$ consequently contributes a locally Lipschitz
function of $p$.

For a degree-two factor write
\[
  z^2+b_p(\zeta)z+d_p(\zeta)
  =(z-c_p(\zeta))^2-w_p(\zeta),
  \qquad
  c_p=-\frac{b_p}{2},\quad w_p=\frac{b_p^2}{4}-d_p.
\]
The functions $c_p,w_p$ depend holomorphically on the coefficient parameters
and are uniformly Lipschitz there.  We use the following elementary symmetric
root estimate.  If $G\in C^2(\C)$ and
\[
  H(c,w)=G(c+s)+G(c-s),\qquad s^2=w,
\]
then $H$ is locally Lipschitz even at $w=0$.  More precisely, if $M_1,M_2$
bound $\|DG\|$ and $\|D^2G\|$ on the relevant compact set, then
\begin{equation}\label{eq:symmetric-root-estimate}
  |H(c,w)-H(c',w')|
  \le 2M_1|c-c'|+\sqrt2\,M_2|w-w'|.
\end{equation}
To see this, choose $s^2=w$ and $t^2=w'$ and change the sign of $t$ so that
$\operatorname{Re}(s\overline t)\ge0$.  Then
$|s+t|\ge(|s|+|t|)/\sqrt2$.  For
$E_c(v)=G(c+v)+G(c-v)$ one has
$\|DE_c(v)\|\le2M_2|v|$, and integration along the segment from $t$ to $s$
gives
\[
 |E_c(s)-E_c(t)|
 \le M_2(|s|+|t|)|s-t|
 \le\sqrt2 M_2|w-w'|.
\]
Varying the center contributes at most $2M_1|c-c'|$, proving
\eqref{eq:symmetric-root-estimate}.

Applying this estimate to the uniformly $C^2$ family $G_\zeta$, the two-root
contribution
\[
  G_\zeta(c_p(\zeta)+s_p(\zeta))
  +G_\zeta(c_p(\zeta)-s_p(\zeta)),
  \qquad s_p(\zeta)^2=w_p(\zeta),
\]
is locally Lipschitz in $p$, uniformly in $\zeta$; no continuous labeling of
the two roots is required.  Integrating in $\zeta$ and summing the fixed finite
partition shows
\[
  |\Psi(p)-\Psi(q)|\le C\|p-q\|
\]
near $p_0$.  The quadratic Chow potential is therefore locally Lipschitz.  This
argument does not assert finiteness of the critical moment
$\int\|p\|^{-1}\,d\lambda$, which indeed diverges at a double hyperplane; the
gain comes from summing the two roots after convolution with the smooth spatial
density.
\end{proof}

The preceding potential estimate now yields the sharp inner-distance bounds.

\begin{theorem}[Sharp inner $W_q$ geometry]
\label{thm:acl-compactness}
Let $n,d\ge1$ and $1\le q\le2$, and let $W_{q,\mathrm{ACL}}^{\rm in}$ be the
inner $q$-Wasserstein distance introduced in \cite{ACL}, with the Fubini--Study
normalization used there.  Then $W_{q,\mathrm{ACL}}^{\rm in}$ induces the usual
projective topology on $P_{n,d}$; in particular
$(P_{n,d},W_{q,\mathrm{ACL}}^{\rm in})$ is compact.  Fix any chordal distance $h$
on $P_{n,d}$.  The constants below may depend on $n,d,q$, and $h$.  The
quantitative bounds are:
\begin{enumerate}
\item if $d=1$, then there is $C_q<\infty$ such that
\begin{equation}\label{eq:acl-holder-linear}
  W_{q,\mathrm{ACL}}^{\rm in}(p,p')\le C_q h(p,p');
\end{equation}
\item if $d\ge2$, then there is $C_q<\infty$ such that
\begin{equation}\label{eq:acl-holder-endpoint}
  W_{q,\mathrm{ACL}}^{\rm in}(p,p')\le C_q h(p,p')^{1/d},
\end{equation}
and the exponent $1/d$ is optimal for each fixed $q\in[1,2]$: for every
$\beta>1/d$ no estimate
$W_{q,\mathrm{ACL}}^{\rm in}(p,p')\le C h(p,p')^\beta$ holds on a neighborhood
of the multiple hyperplane point $[z_0^d]$.
\end{enumerate}
\end{theorem}

\begin{proof}
Write $g_{\mathrm{FS}}^{(0)},\omega_0$ for the ambient Fubini--Study structure used
in this paper, so that $\int_{\PP^1}\omega_0=1$, and let $W_{2,0}$ be the
corresponding Wasserstein distance.  The Hopf-quotient Fubini--Study metric used
in \cite{ACL} has $\PP^1$-area $\pi$: in an affine coordinate $z$ on $\PP^1$,
\[
  ds_{\rm ACL}^2=\frac{|dz|^2}{(1+|z|^2)^2},
  \qquad
  \operatorname{Area}_{\rm ACL}(\PP^1)
  =\int_{\C}\frac{dx\,dy}{(1+|z|^2)^2}=\pi.
\]
The two Fubini--Study metrics are $U(n+1)$-invariant and compatible with the same
complex structure, so
\begin{equation}\label{eq:acl-fs-scale}
  g_{\mathrm{FS}}^{\rm ACL}=\pi g_{\mathrm{FS}}^{(0)},
  \qquad
  W_{2,\mathrm{ACL}}=\sqrt{\pi}\,W_{2,0}.
\end{equation}
The normalized hypersurface measures are the same in both conventions: on the
smooth locus the constant rescaling cancels between volume and total mass, and
continuity extends the equality across the discriminant.  Write
\[
  \mathcal M_{\rm hyp}:P_{n,d}\longrightarrow\mathcal P_2(\PP^n),
  \qquad
  \mathcal M_{\rm hyp}(p)=\mu_p.
\]
The admissible path classes agree, and the corresponding constrained action
distances satisfy
\begin{equation}\label{eq:acl-inner-scale-pre}
  W_{2,\mathrm{ACL}}^{\rm in}=\sqrt{\pi}\,W_{2,0}^{\rm in},
\end{equation}
where
\begin{equation}\label{eq:w20-inner-definition}
 \bigl(W_{2,0}^{\rm in}(p,q)\bigr)^2:=
 \inf_{\substack{\gamma\in C([0,1],P_{n,d}),\ \gamma(0)=p,\ \gamma(1)=q\\
                  \mathcal M_{\rm hyp}\circ\gamma\in AC^2([0,1],W_{2,0})}}
 \int_0^1
 \left|\frac d{dt}\mathcal M_{\rm hyp}(\gamma(t))\right|_{W_{2,0}}^2\,dt.
\end{equation}

Specializing \cref{cor:chow-inner-compactness} to $X=P_{n,d}$, whose
normalization is the identity, identifies \eqref{eq:w20-inner-definition} with
$D_X$.  It follows that $W_{2,0}^{\rm in}$ is a finite metric inducing the
projective topology, $(P_{n,d},W_{2,0}^{\rm in})$ is compact, and
\begin{equation}\label{eq:w20-bar-d}
  W_{2,0}^{\rm in}(p,p')\le \bar d(p,p').
\end{equation}

If $d=1$, the discriminant is empty and $B=P_{n,1}$.  The smooth positive
definite metric $g_B$ is uniformly comparable on the compact space $P_{n,1}$
with the standard Fubini--Study metric, so
\begin{equation}\label{eq:smooth-hypersurface-linear}
  \bar d(p,p')=d_B(p,p')\le C h(p,p').
\end{equation}
If $d\ge2$, \cref{lem:hyp-critical-potential} gives local
$C^{0,2/d}$ potentials on the smooth normal space $P_{n,d}$.  Applying
\cref{lem:localdistance} on the smooth locus
$B=P_{n,d}\setminus\Delta_{n,d}$, as in the proof of
\cref{cor:smooth-normalization-holder}, and then using a finite coordinate cover
gives
\[
  \bar d(a,b)=d_B(a,b)\le C h(a,b)^{1/d},
  \qquad a,b\in B.
\]
For arbitrary $p,p'\in P_{n,d}$, choose $a_j,b_j\in B$ with
$a_j\to p$ and $b_j\to p'$.  Passing to the completion by
\eqref{eq:limitmetric}, and using continuity of the chordal distance, yields
\begin{equation}\label{eq:hypersurface-bar-d-endpoint}
  \bar d(p,p')\le C h(p,p')^{1/d}
  \qquad (p,p'\in P_{n,d}).
\end{equation}
Combining these bounds with \eqref{eq:w20-bar-d} and
\eqref{eq:acl-inner-scale-pre} proves the stated estimates for $q=2$.

Fix now $1\le q\le2$.  If $W_{q,0}$ denotes ambient $q$-Wasserstein distance
for our normalization, then \eqref{eq:acl-fs-scale} gives
$W_{q,\mathrm{ACL}}=\sqrt\pi\,W_{q,0}$.  Moreover,
$W_{q,\mathrm{ACL}}\le W_{2,\mathrm{ACL}}$ on probability measures, every
$W_2$-absolutely continuous measure curve is $W_q$-absolutely continuous, and
$|\dot\mu|_q\le|\dot\mu|_2$ almost everywhere.  Since the parameter interval has
unit length, $\|v\|_{L^q}\le\|v\|_{L^2}$ for $q\le2$; taking infima gives
\begin{equation}\label{eq:wq-less-w2-inner}
  W_{q,\mathrm{ACL}}^{\rm in}(p,p')
  \le W_{2,\mathrm{ACL}}^{\rm in}(p,p').
\end{equation}
Therefore projective convergence implies convergence in the inner $W_q$ distance.
Conversely,
\[
  W_{q,\mathrm{ACL}}(\mu_p,\mu_{p'})
  \le W_{q,\mathrm{ACL}}^{\rm in}(p,p').
\]
Compactness of $\PP^n$ makes ambient $W_q$ convergence equivalent to weak
convergence, and the cycle measure map is
a topological embedding of compact $P_{n,d}$ into the probability measures.
Together, these implications show that $W_{q,\mathrm{ACL}}^{\rm in}$ induces the
projective topology and is compact for every $q\in[1,2]$.

For every $d\ge2$, sharpness follows from the family below, with
$q\in[1,2]$ fixed:
\[
  p_t=z_0^d-tz_1^d,
  \qquad p_0=z_0^d,
  \qquad a=t^{1/d},
  \qquad t>0.
\]
For any fixed chordal metric on the coefficient projective space,
$h([p_t],[p_0])\asymp t$ as $t\downarrow0$.  For $t>0$ the hypersurface cycle
splits as
\[
  Z(p_t)=\sum_{k=1}^d H_k,
  \qquad
  H_k=\{z_0=a\zeta_k z_1\},
  \qquad \zeta_k^d=1,
\]
whereas $p_0$ represents the hyperplane $H_0=\{z_0=0\}$ with multiplicity
$d$.  The associated measures are
\[
  \mu_{p_t}=\frac1d\sum_{k=1}^d\mu_{H_k},
  \qquad
  \mu_{p_0}=\mu_{H_0}.
\]
Every coupling with second marginal supported on $H_0$ pays at least the
$q$th power of the distance to $H_0$.  In the Fubini--Study metric normalized
by $\int_{\PP^1}\omega=1$,
\[
  \operatorname{dist}_0([z],H_0)
  =\frac1{\sqrt\pi}\arcsin\frac{|z_0|}{\|z\|}
  \ge\frac1{\sqrt\pi}\frac{|z_0|}{\|z\|}.
\]
Choose the orthonormal basis of the linear subspace underlying $H_k$ whose
first vector is
\[
  e'_1=\frac{a\zeta_k e_0+e_1}{\sqrt{1+a^2}}
\]
and whose remaining vectors are $e_2,\ldots,e_n$.  In the corresponding
homogeneous coordinates $[w_1:\cdots:w_n]$, put
$u=|w_1|^2/(|w_1|^2+\cdots+|w_n|^2)$.  Then
\[
  \frac{|z_0|}{\|z\|}=\frac{a}{\sqrt{1+a^2}}\sqrt u.
\]
Unitary invariance on $H_k\simeq\PP^{n-1}$ implies
$\int u\,d\mu_{H_k}=1/n$; since $q/2\le1$, one has $u^{q/2}\ge u$.  Therefore
\[
 \int_{H_k}\operatorname{dist}_0(x,H_0)^q\,d\mu_{H_k}(x)
 \ge \frac{a^q}{\pi^{q/2}n(1+a^2)^{q/2}}.
\]
Averaging over the $d$ components and using
$W_{q,\mathrm{ACL}}=\sqrt\pi\,W_{q,0}$ gives
\[
  W_{q,\mathrm{ACL}}^{\rm in}(p_t,p_0)
  \ge W_{q,\mathrm{ACL}}(\mu_{p_t},\mu_{p_0})
  \ge \frac{a}{n^{1/q}\sqrt{1+a^2}}
  \asymp t^{1/d}.
\]
Any upper estimate with exponent $\beta>1/d$ would imply
$t^{1/d}\lesssim t^\beta$ as $t\downarrow0$, which is impossible.  Hence
$1/d$ is the optimal uniform exponent for every fixed $q\in[1,2]$ and every
$d\ge2$.
\end{proof}

\subsection{Complete intersections}

\begin{corollary}\label{cor:complete-intersections}
Fix a multidegree $\mathbf d=(d_1,\ldots,d_k)$ with $1\le k<n$, let
$m=n-k$ and $D=\prod_i d_i$, and let $B=\cC_{\mathbf d}^{\rm sm}$ be the smooth
complete intersection cycle locus.  If
$X=\overline B^{\,\Chow}_{\rm red}$, then all conclusions of
\cref{thm:main} hold for $X$.
\end{corollary}

\begin{proof}
Let $\mathcal U_{\mathbf d}^{\rm sm}$ be the Zariski-open locus in the product
of the spaces of homogeneous forms of degrees $d_1,\ldots,d_k$, consisting
of regular sequences whose common zero scheme is smooth of codimension $k$.
Its image in the Hilbert scheme is precisely the locally closed locus
$H_{\mathbf d}^{\rm sm}$ of smooth complete intersections of this multidegree.
For $Z\in H_{\mathbf d}^{\rm sm}$, the conormal sequence identifies
\[
 I_Z/I_Z^2\simeq\bigoplus_{j=1}^k\mathcal O_Z(-d_j),
 \qquad
 N_{Z/\PP^n}\simeq\bigoplus_{j=1}^k\mathcal O_Z(d_j).
\]
Standard deformation theory for complete intersections shows that the Hilbert
scheme is smooth along this locus and that
$T_ZH_{\mathbf d}^{\rm sm}=H^0(Z,N_{Z/\PP^n})$.  Equivalently, local embedded deformations come from deforming the regular sequence,
modulo changes of generators.
The parameter space $\mathcal U_{\mathbf d}^{\rm sm}$ is a nonempty open
subset of an irreducible affine space, hence is itself irreducible.  Its image
$H_{\mathbf d}^{\rm sm}$ in the Hilbert scheme is irreducible as well and is
therefore connected.  See \cite[Section~2.2]{DL} for this parameter space and
Hilbert scheme description.

Every member of $H_{\mathbf d}^{\rm sm}$ is smooth and hence normal.  In
characteristic zero the Hilbert--Chow morphism is an isomorphism over the locus
parametrizing normal equidimensional subschemes; see
\cite[Part~IV, Corollary~12.9 and Section~17]{Rydh08}.  Since the relative Chow functor used
there agrees with the classical Chow variety on reduced bases, the Hilbert--Chow
morphism restricts to an isomorphism
\[
 H_{\mathbf d}^{\rm sm}\xrightarrow{\sim} B
\]
onto its image.  The Hilbert--Chow isomorphism identifies $B$ as a smooth irreducible locally closed subset of the Chow variety.
Its reduced closure $X=\overline B^{\,\Chow}_{\rm red}$ is irreducible, and
$B$ is Zariski open in $X$.  The locus $B$ is admissible in the sense of \cref{def:admissible}, and the hypotheses of \cref{thm:main} are satisfied.
\end{proof}
\section{Wasserstein geometry on singular strata}\label{sec:singular-strata}

Metric normalization identifies the boundary points of the completion, but the
global action theorem also requires an infinitesimal description along singular
loci.  We therefore stratify $Y$ so that the resolved component families vary
smoothly.  On each stratum their normal motion again gives the exact Wasserstein
speed, providing the local dynamical input for \cref{sec:recovery}.

\subsection{Resolved families on algebraic strata}

\begin{lemma}\label{lem:strata}
There is a finite partition of $Y$ into smooth locally closed complex algebraic
strata $S$ such that, on every positive-dimensional stratum and locally in the
analytic topology:
\begin{enumerate}
\item $\chi\nu|_S$ is an immersion;
\item there are smooth proper holomorphic submersions
      $p_i:\widetilde Z_i\to S$ with connected smooth $m$-dimensional fibers,
      holomorphic maps $e_i:\widetilde Z_i\to\PP^n$, and fixed integers
      $a_i>0$, such that
\begin{equation}\label{eq:resolved-cycle}
 \nu(s)=\sum_i a_i(e_i)_*[\widetilde Z_{i,s}]
\end{equation}
as effective cycles for every $s\in S$.  The images are the distinct irreducible
reduced components, and each $e_{i,s}$ is birational onto its image and is an
isomorphism over a dense regular open subset.
\end{enumerate}
The dense isomorphism loci in (2) may be chosen as relative open subsets of the
resolved families, with fiberwise proper analytic complements.  The presentation
may first be constructed after a finite \'etale cover
of $S$; on sufficiently small analytic neighborhoods, inverse branches of the
cover recover the stated local data on $S$ itself.
\end{lemma}

\begin{proof}
Apply \cref{lem:generic-resolved} to the finite morphism
$\chi\nu:Y\to\chi(X)\subset\PP(V_{\rm Ch})$.  The lemma supplies the required
data on a dense smooth open subset of each
positive-dimensional irreducible component, where $\chi\nu$ is immersive.
Removing these pieces lowers the dimension of every remaining irreducible
component.  Repeating the construction on the finitely many components of the
remainder terminates by Noetherian induction.  Refining
overlaps by locally closed differences produces a finite partition by smooth locally
closed strata.  The local resolved presentations and finite \'etale descent are
those of
\cref{lem:generic-resolved}; the remaining zero-dimensional pieces are taken as
individual strata.
\end{proof}

\begin{lemma}\label{lem:stratum-lift}
Let $\rho:R\to Y$ be the projective resolution fixed in
\cref{sec:normalization}.  After refining the stratification in
\cref{lem:strata}, every positive-dimensional stratum $S$ admits, locally in the
analytic topology, a holomorphic map
\[
 h:S\longrightarrow R,\qquad \rho\circ h=\id_S.
\]
\end{lemma}

\begin{proof}
Fix an integral stratum closure $W$.  The generic fiber of
$R\times_YW\to W$ is a nonempty projective scheme over $K=\C(W)$.  Choose a
closed point of that generic fiber, with finite residue extension $K'/K$, and let
$g:W'\to W$ be the normalization in $K'$.  The induced $K'$-rational point spreads
to a section of $R\times_YW'\to W'$ over a dense open subset.  Using the
saturated descent from \cref{lem:strata}, shrink $W$ to a nonempty open $U$ whose
full inverse image $U'=g^{-1}(U)$ lies in the domain of the section.  As in \cref{lem:generic-resolved}, the map $g$ is finite.  Separability of
$K'/K$, together with the generic smoothness criterion
\cite[\href{https://stacks.math.columbia.edu/tag/07ND}{Tag~07ND}]{Stacks}, then
allows a further saturated shrinking for which $U'\to U$ is finite \'etale.  On a
sufficiently small simply connected analytic neighborhood in $U$, an inverse branch
of this cover pulls the section back to the required holomorphic lift.  Refining by
the lower-dimensional complements and applying Noetherian induction produces a finite
stratification compatible with \cref{lem:strata}.
\end{proof}

\subsection{Wasserstein geometry on a stratum}

Work with one local presentation from \cref{lem:strata}, and set
\[
 \mathsf V_D:=DV_m=\frac{D}{m!},\qquad
 \omega_i=e_i^*\omega,\qquad
 dV_{i,s}=\frac{\omega_{i,s}^m}{m!}.
\]
\begin{lemma}\label{lem:resolved-volume}
For every $i$ and $s\in S$,
\begin{equation}\label{eq:pushforward-volume}
 (e_{i,s})_*\!\left(\frac{(e_{i,s}^*\omega)^m}{m!}\right)
 =\frac{\omega^m|_{C_{i,s}^{\rm reg}}}{m!}
\end{equation}
as Radon measures, where $C_{i,s}$ is the reduced irreducible image cycle.
Equation \eqref{eq:resolved-measure} agrees with the normalized cycle measure
\eqref{eq:measure}.
\end{lemma}
\begin{proof}
There are dense Zariski-open sets
$U_i\subset\widetilde Z_{i,s}$ and
$V_i\subset C_{i,s}^{\rm reg}$ on which $e_{i,s}:U_i\to V_i$ is a biholomorphism.
The complements are proper complex analytic subsets and are negligible for real
$2m$-dimensional measure.  On $U_i$ the two volume forms agree by change of
variables.  Neither the exceptional set on the source nor the singular complement
on the target contributes mass, so the Radon measures coincide.
\end{proof}

With this notation, \cref{lem:resolved-volume} identifies the cycle measure as
\begin{equation}\label{eq:resolved-measure}
 \mu_s=\mathsf V_D^{-1}\sum_i a_i(e_{i,s})_*(dV_{i,s}).
\end{equation}
On the regular portion where $e_{i,s}$ is an embedding, let $N_{i,v}$ be the normal
component of $de_i(H_v)$ for a lift $H_v$ of $v\in T_sS$.  Define
\begin{equation}\label{eq:stratum-current}
 \Omega_S=
 \mathsf V_D^{-1}\sum_i a_i(p_i)_*\frac{\omega_i^{m+1}}{(m+1)!},
 \qquad
 g_S(v,w)=\Omega_S(v,Jw).
\end{equation}

A curve in a stratum is called \emph{absolutely continuous} when it is absolutely
continuous for one, equivalently any, smooth auxiliary Riemannian metric.  Its
$g_S$-energy is $\int g_S(\dot s,\dot s)$.

\begin{proposition}\label{prop:resolved-speed}
The form $\Omega_S$ is smooth, closed and nonnegative.  For every real
$v\in T_sS$,
\begin{equation}\label{eq:stratum-energy}
 g_S(v,v)
 =
 \mathsf V_D^{-1}\sum_i a_i\int_{\widetilde Z_{i,s}}
 |N_{i,v}|^2\,dV_{i,s}.
\end{equation}
For every $\phi\in C^\infty(\PP^n)$,
\begin{equation}\label{eq:firstvariation}
 d\!\left(\int\phi\,d\mu_s\right)(v)
 =
 \mathsf V_D^{-1}\sum_i a_i
 \int_{\widetilde Z_{i,s}}d\phi(N_{i,v})\,dV_{i,s}.
\end{equation}
For every absolutely continuous curve $s:[0,T]\to S$, set $\mu_t=\mu_{s(t)}$.
Then $\mu_\cdot\in AC^1([0,T],W_2)$ and
\begin{equation}\label{eq:singular-speed}
 |\dot\mu_t|_{W_2}^2
 =g_S(\dot s(t),\dot s(t))
 \quad\text{for a.e. }t.
\end{equation}
If $s$ has finite $g_S$-energy, then $\mu_\cdot\in AC^2([0,T],W_2)$.
\end{proposition}

\begin{proof}
Smoothness and closedness of \eqref{eq:stratum-current} follow by fiber integration
along the smooth proper submersions $p_i$.  On the locus of full rank, split $de_i(H_v)$ into tangent and normal parts.  The complex tangent and normal subspaces are $\omega$-orthogonal, so evaluation of
$\omega_i^{m+1}/(m+1)!$ on $H_v,JH_v$ and vertical vectors equals
$|N_{i,v}|^2\omega_{i,s}^m/m!$, proving \eqref{eq:stratum-energy}.

For the first variation, differentiate along the compact resolved fiber.  Cartan's
formula and $d\omega_i=0$ reduce the derivative to
\begin{align*}
d\!\left((p_i)_*((\phi\circ e_i)\omega_i^m/m!)\right)(v)
&=
\int_{\widetilde Z_{i,s}}
\iota_{H_v}\left(e_i^*d\phi\wedge\omega_i^m/m!\right).
\end{align*}
On the immersive locus write
$de_i(H_v)=T_{i,v}+N_{i,v}$ with $T_{i,v}$ tangent to the component image.  Since
$\iota_{H_v}\omega_i$ restricts to $\iota_{T_{i,v}}\omega$ on the fiber, the
contraction identity reads
\[
 \left.\iota_{H_v}\!\left(e_i^*d\phi\wedge\frac{\omega_i^m}{m!}\right)
 \right|_{\widetilde Z_{i,s}}
 =\bigl[d\phi(T_{i,v}+N_{i,v})-d\phi(T_{i,v})\bigr]dV_{i,s}
 =d\phi(N_{i,v})\,dV_{i,s}.
\]
The tangential contribution cancels.  The locus where the rank drops is a proper analytic
subset of the smooth resolved fiber and has zero smooth volume; we set
$N_{i,v}=0$ there.  Removing this null set leaves the integral unchanged.  Summing over the components proves \eqref{eq:firstvariation}.

Distinct component images meet only in lower complex dimension.  Their normal
fields piece together almost everywhere into a unique $L^2(\mu_s)$
field.  Let $\Sigma_i$ be the union in the resolved fiber of the locus where the rank drops, the
exceptional locus, and the preimages of intersections with the other component
images.  Each $\Sigma_i$ is contained
in a proper analytic subset and is $dV_{i,s}$-null.  Since
$|N_{i,v}|^2\,dV_{i,s}$ is a finite measure by \eqref{eq:stratum-energy}, choose
shrinking neighborhoods $O_{i,k}\supset\Sigma_i$ such that
\[
  \int_{O_{i,k}}|N_{i,v}|^2\,dV_{i,s}\longrightarrow0,
\]
and smooth cutoffs $0\le\chi_{i,k}\le1$ which vanish near $\Sigma_i$ and equal
one off $O_{i,k}$.  Because $\chi_{i,k}$ vanishes on a neighborhood of
$\Sigma_i$, its support is a compact subset of the relative isomorphism locus
from \cref{lem:strata}.  Its image is therefore contained in a relatively
compact embedded open subset of $C_{i,s}^{\rm reg}$; for fixed $k$, these
embedded neighborhoods may be chosen disjoint for distinct components.  Transfer
$\chi_{i,k}N_{i,v}$ through this identification.  In a tubular neighborhood of
such an embedded open subset, with normal coordinate $\xi$, choose a
normal cutoff $\zeta_{i,k}$ and set
\[
 \varphi_{i,k}(z,\xi)
 =\chi_{i,k}(z)\,g_{\mathrm{FS}}(N_{i,v}(z),\xi)\,\zeta_{i,k}(z,\xi),
\]
where $z$ now denotes the corresponding point of the regular component.
Because $\varphi_{i,k}$ vanishes along the zero section, its tangential differential
there is zero, while its normal differential is $\chi_{i,k}N_{i,v}$.  By construction,
$\nabla\varphi_{i,k}|_{C_{i,s}^{\rm reg}}=\chi_{i,k}N_{i,v}$.  The tubular
neighborhoods may be chosen disjoint, so these ambient gradients can be summed.
By the preceding choice of $O_{i,k}$,
\[
 \sum_i a_i\int |(1-\chi_{i,k})N_{i,v}|^2\,dV_{i,s}\longrightarrow0.
\]
The approximation places $N_v$ in the $L^2(\mu_s)$-closure of ambient gradients,
which is the Wasserstein tangent space $T_{\mu_s}\mathcal P_2$.
For fixed $(s,v)$, \eqref{eq:firstvariation} determines the corresponding
functional associated with the continuity equation.  If $V$ is any other $L^2(\mu_s)$ field inducing
the same functional, then for every smooth $\phi$
\[
  \int_{\PP^n}\langle\nabla\phi,V-N_v\rangle\,d\mu_s=0.
\]
The last identity makes $V-N_v$ orthogonal to $T_{\mu_s}\mathcal P_2$, whereas
$N_v\in T_{\mu_s}\mathcal P_2$.  The orthogonal projection of any representative of the same functional onto the tangent space is $N_v$, which is the unique minimal-norm representative.

Let $s:[0,T]\to S$ be absolutely continuous.
Cover its compact image by finitely many charts carrying resolved presentations
and partition the parameter interval accordingly; the argument may be carried out
on each piece separately.
Fix a smooth auxiliary Riemannian metric $h$ there.
On the relevant compact subset, smoothness of $g_S$ implies
\[
  \sqrt{g_S(\dot s(t),\dot s(t))}\le C|\dot s(t)|_h
  \quad\text{for a.e. }t.
\]
In a local smooth trivialization of the resolved families,
$(s,v,z)\mapsto N_{i,v}(z)$ is smooth on the locus of full rank and linear in $v$;
extend it by zero on the analytic locus where the rank drops.  The construction in
\cref{lem:strata} supplies a relative open isomorphism locus on which
$(p_i,e_i)$ has a holomorphic inverse onto its image.  Transfer the field through
this inverse and set it equal to zero on the complementary $\mu_s$-null set.  The
resulting fields $N_t:=N_{\dot s(t)}$ form a Borel velocity field along $\mu_t$.  By \eqref{eq:stratum-energy},
\[
  \int_0^T\|N_t\|_{L^2(\mu_t)}\,dt
  =\int_0^T\sqrt{g_S(\dot s(t),\dot s(t))}\,dt<\infty.
\]
For every $\psi\in C_c^\infty((0,T)\times\PP^n)$, integrating the chain rule
identity from \eqref{eq:firstvariation} over time produces
\[
  \int_0^T\!\int_{\PP^n}
  \bigl(\partial_t\psi+\langle\nabla_x\psi,N_t\rangle\bigr)\,d\mu_t\,dt=0.
\]
The integrated identity is precisely the distributional continuity equation for
$(\mu_t,N_t)$.
For $0\le s\le t\le T$, the standard estimate for such solutions reads
\[
  W_2(\mu_s,\mu_t)
  \le \int_s^t\|N_r\|_{L^2(\mu_r)}\,dr.
\]
Put
\[
  v(t):=\|N_t\|_{L^2(\mu_t)},
  \qquad
  \beta(t):=\int_0^t(1+v(r))\,dr,
  \qquad
  \tau:=\beta^{-1}.
\]
The map $\beta$ is strictly increasing and absolutely continuous.  Define
$\widetilde\mu_s:=\mu_{\tau(s)}$ and rescale the velocity by
\[
  \widetilde N_s
  :=\frac{N_{\tau(s)}}{1+v(\tau(s))}.
\]
The time change formula shows that $(\widetilde\mu_s,\widetilde N_s)$ satisfies the distributional
continuity equation, while
\[
  \int_0^{\beta(T)}\|\widetilde N_s\|_{L^2(\widetilde\mu_s)}^2\,ds
  =\int_0^T\frac{v(t)^2}{1+v(t)}\,dt
  \le\int_0^T v(t)\,dt<\infty.
\]
Scalar rescaling preserves the pointwise minimality established above.  The
$AC^2$ minimal-velocity characterization in Wasserstein space therefore gives
\[
  |\dot{\widetilde\mu}_s|_{W_2}
  =\|\widetilde N_s\|_{L^2(\widetilde\mu_s)}
  \quad\text{for a.e. }s.
\]
Because $\mu_t=\widetilde\mu_{\beta(t)}$, the metric chain rule gives
\[
  |\dot\mu_t|_{W_2}=v(t)=\|N_t\|_{L^2(\mu_t)}
  \quad\text{for a.e. }t.
\]
The metric estimate places $\mu_\cdot$ in $AC^1([0,T],W_2)$, and
\eqref{eq:stratum-energy} becomes \eqref{eq:singular-speed}.  If $s$ has finite
$g_S$-energy, then $v\in L^2(0,T)$ and the same continuity-equation
characterization places $\mu_\cdot$ in $AC^2([0,T],W_2)$.  We use here the
corresponding formulation for compact Riemannian manifolds in
\cite[Theorem~2.4 and Proposition~2.6]{ACL}, together with the metric framework
of \cite[Chapter~8]{AGS}.
\end{proof}

\begin{proposition}\label{prop:restriction-curvature}
On every stratum,
\begin{equation}\label{eq:curvature-on-S}
 \Omega_S=
 \frac1{D(m+1)}(\chi\nu|_S)^*\cT.
\end{equation}
On every positive-dimensional stratum, $g_S$ is positive definite.
\end{proposition}

\begin{proof}
Work on one analytic neighborhood carrying the resolved presentation.  Let
$p_S,p_G$ denote the projections from $S\times G$, and let
$\mathcal D_S\subset S\times G$ be the pullback of the universal Chow divisor
under $\chi\nu|_S$.  For each $i$, form the incidence pullback
\[
 I_i=\widetilde Z_i\times_{\PP^n} I
\]
using $e_i:\widetilde Z_i\to\PP^n$ and $a:I\to\PP^n$, and let
$r_i:I_i\to S\times G$ be the induced proper map.  By functoriality of relative cycles under proper pushforward, the relative
incidence divisor associated with the total cycle
$\sum_i a_i(e_i)_*[\widetilde Z_{i,s}]$ is the sum of the proper pushforwards of the component incidence cycles, with precisely these multiplicities.  The resulting identity, as cycles (or
equivalently as integration currents), is
\begin{equation}\label{eq:total-incidence-stratum}
 [\mathcal D_S]=\sum_i a_i(r_i)_*[I_i].
\end{equation}

Apply the Poincar\'e--Lelong pushforward from \cref{prop:curvature} to the local
Chow potential pulled back to $S$.  Combining
\eqref{eq:total-incidence-stratum} with Fubini and \cref{lem:crofton} identifies the pullback as
\[
 (\chi\nu|_S)^*\cT
 =(p_S)_*\bigl([\mathcal D_S]\wedge p_G^*d\nu_G\bigr)
 =\sum_i a_i(p_i)_*e_i^*\bigl(a_*b^*d\nu_G\bigr)
 =\sum_i a_i(p_i)_*e_i^*\omega^{m+1}.
\]
Since
$\mathsf V_D^{-1}(m+1)!^{-1}=1/[D(m+1)]$, this is
\eqref{eq:curvature-on-S}.

Suppose $g_S(v,v)=0$.  By \eqref{eq:stratum-energy}, every normal deformation
$N_{i,v}$ vanishes almost everywhere on the locus of full rank, and smoothness makes
this vanishing pointwise.  Fix a component and choose a local holomorphic
representative $R_i(s)$ of its Chow form, with derivative
$\dot R_i=dR_i(s)[v]$.  At a general incidence pair $(x,L)$, the point $x$ lies
in $C_{i,s}^{\rm reg}$, the intersection with $L$ is unique, and
$T_xC_{i,s}\cap T_xL=\{0\}$.  The corresponding incidence map is then locally
biholomorphic onto the smooth locus of the component Chow divisor.  In the local
graph description of the family, the vanishing of $N_{i,v}$ makes the first-order
graph deformation zero.  Transporting this infinitesimal deformation through the
local incidence isomorphism shows that the normal first-order displacement of the
component Chow divisor vanishes, so
$\dot R_i$ vanishes on a dense regular open subset of that divisor, and then on the
whole divisor by holomorphicity.

The generic incidence argument from \cref{prop:curvature} shows that the Chow
divisor of a reduced irreducible component is reduced and irreducible.  Hence
$\dot R_i$ is divisible by $R_i$; because both are sections of the same line
bundle $\mathcal O_G(D_i)$, the quotient is a global holomorphic function on the
compact connected Grassmannian $G$ and must be constant.  Writing
$\dot R_i=\lambda_iR_i$ means precisely
\[
  d\kappa_i(v)=0.
\]
Applying the argument to every component and using the local identity between the
weighted total cycle morphism $\sum_i a_i\kappa_i$ and $\chi\nu|_S$ shows that
$d(\chi\nu|_S)(v)=0$.  Immersivity of $\chi\nu|_S$ then forces $v=0$, so
$g_S$ is positive definite.
\end{proof}

\subsection{Distance gradients on strata}

Recall the quadratic form $Q(F)$ from \eqref{eq:Qdef}.  By
\eqref{eq:Qunit}, with our convention $\Omega(v,Jv)=g(v,v)$,
\begin{equation}\label{eq:sharp-factor}
 Q(F)\le\Omega
 \quad\Longleftrightarrow\quad
 |dF|_g\le1.
\end{equation}

\begin{lemma}\label{lem:trace}
Let $u$ be continuous psh on a coordinate domain $U$ and let
$F\in C^0(U)\cap W_{\rm loc}^{1,2}(U)$ be real-valued with
\[
 Q(F)\le dd^cu.
\]
For every holomorphic map $h:W\to U$ from a complex manifold $W$ one has
\[
 F\circ h\in W_{\rm loc}^{1,2}(W),
 \qquad
 Q(F\circ h)\le dd^c(u\circ h)
\]
as positive currents on $W$.  If $dd^c(u\circ h)$ is represented by a smooth
positive definite $(1,1)$-form $\Omega_h$, then $F\circ h$ is locally
$1$-Lipschitz for the Riemannian length metric of $\Omega_h$.
\end{lemma}

\begin{proof}
Regularize $F$ and $u$ in the ambient coordinates, on slightly smaller domains:
$F_\eps=F*\rho_\eps$ and $u_\eps=u*\rho_\eps$.  Matrix Jensen's inequality implies
\[
 Q(F_\eps)\le Q(F)*\rho_\eps\le dd^cu_\eps.
\]
Holomorphic pullback preserves this inequality, so
\[
 Q(F_\eps\circ h)\le dd^c(u_\eps\circ h).
\]
Let
$W'\Subset W_0\Subset W$ be coordinate balls, let $\omega_0$ be the Euclidean
K\"ahler form on $W_0$, and choose $\chi\in C_c^\infty(W_0)$ with
$0\le\chi\le1$ and $\chi\equiv1$ on $W'$.  For all sufficiently small $\eps$,
\begin{align*}
 \int_{W'} |d(F_\eps\circ h)|_{\omega_0}^2\,dV_{\omega_0}
 &\le C\int_{W_0}\chi\,dd^c(u_\eps\circ h)\wedge\omega_0^{k-1}\\
 &= C\int_{W_0}(u_\eps\circ h)\,dd^c\chi\wedge\omega_0^{k-1}
 \le C_{W',W_0},
\end{align*}
where $k=\dim_\C W$.  The last bound is uniform because
$u_\eps\circ h\to u\circ h$ uniformly on $\overline{W_0}$.  Since
$F_\eps\circ h\to F\circ h$ uniformly as well, this ensures a uniform local
$W^{1,2}$ bound.  Passing to a weakly convergent subsequence, fix a smooth strongly
positive $(k-1,k-1)$-test form $\eta$.  The functional obtained by pairing
$Q(\cdot)$ with $\eta$ is a nonnegative convex quadratic $L^2$-functional of the
weak gradient, so it is weakly lower semicontinuous.  Since
$dd^c(u_\eps\circ h)\to dd^c(u\circ h)$ distributionally, the tested inequalities
pass to the limit for every such $\eta$, giving
\[
 Q(F\circ h)\le dd^c(u\circ h).
\]
The limit is independent of the subsequence because of uniform convergence.
If $dd^c(u\circ h)$ is represented by a smooth positive definite form
$\Omega_h$, \eqref{eq:sharp-factor} and the Sobolev-to-Lipschitz property for a
smooth Riemannian metric prove the final assertion.
\end{proof}

For $p\in B$, consider the completion distance function pulled back to $R$,
$F_p(r)=\bar d(p,\rho(r))$.  The current $T_R$ is positive and closed on $R$,
and it is a smooth K\"ahler form on $B$.  By \cref{lem:weakgradient},
\begin{equation}\label{eq:sharp-distance-current}
 Q(F_p)\le T_R.
\end{equation}

\begin{proposition}\label{prop:stratum-distance}
For every absolutely continuous curve $s:[0,T]\to S$, set $\mu_t=\mu_{s(t)}$.  Then
\begin{equation}\label{eq:all-stratum-speeds}
 |\dot s(t)|_{\bar d}
 =|\dot\mu_t|_{W_2}
 =\sqrt{g_S(\dot s(t),\dot s(t))}
 \quad\text{for a.e. }t.
\end{equation}
\end{proposition}

\begin{proof}
Use the local holomorphic lift $h:S\to R$ from \cref{lem:stratum-lift}.  By
\cref{prop:restriction-curvature}, $h^*T_R=\Omega_S$.  Apply
\cref{lem:trace} to $F_p$ and a local potential of $T_R$.  The function
$s\mapsto\bar d(p,s)$ is $1$-Lipschitz for the $g_S$ length metric.  Let
$D_0\subset B$ be countable and dense in $(Y,\bar d)$.  Then, for nearby points
$s_0,s_1$ in the same stratum chart,
\[
 \bar d(s_0,s_1)
 =\sup_{p\in D_0}\bigl|\bar d(p,s_0)-\bar d(p,s_1)\bigr|
 \le d_{g_S}(s_0,s_1).
\]
We therefore have $|\dot s(t)|_{\bar d}\le |\dot s(t)|_{g_S}$ almost everywhere.  The global inequality
$W_2\le\bar d$ implies
$|\dot\mu_t|_{W_2}\le |\dot s(t)|_{\bar d}$, while
\cref{prop:resolved-speed} identifies the ambient speed with $|\dot s(t)|_{g_S}$.
\end{proof}

\section{Global action formula}\label{sec:recovery}

We now pass from the stratumwise speed identity to arbitrary continuous paths in
the normalization.  Set
\[
  \mathcal M:Y\longrightarrow\mathcal P_2(\PP^n),
  \qquad
  \mathcal M(y)=\mu_{\nu(y)}.
\]
The map $\mathcal M$ is $1$-Lipschitz because $W_2\le\bar d$.

\begin{theorem}[Ambient action identity]\label{thm:full-action}
Let $\gamma:[0,T]\to Y$ be continuous and set $\mu_t=\mathcal M(\gamma(t))$.  For every
$1\le p\le\infty$,
\begin{equation}\label{eq:ACequiv}
 \gamma\in AC^p([0,T],\bar d)
 \quad\Longleftrightarrow\quad
 \mu_\cdot\in AC^p([0,T],W_2).
\end{equation}
In either case,
\begin{equation}\label{eq:full-action-speed}
 |\dot\gamma(t)|_{\bar d}
 =
 |\dot\mu_t|_{W_2}
 \quad\text{for a.e. }t.
\end{equation}
\end{theorem}

\begin{remark}\label{rem:continuity-essential}
Continuity is what keeps track of the normalization branch.  A jump between two
distinct points in one fiber of $\mathcal M$ has constant measure image, but it
is not $\bar d$-absolutely continuous.
\end{remark}

\Cref{prop:stratum-distance} already gives the speed identity along every
individual stratum.  The remaining difficulty is that a continuous path may meet
different strata on highly fragmented sets of times.  The moment-coordinate and
stitching lemmas below control these transitions; they are assembled in
\cref{sec:action-proof} to prove \cref{thm:full-action}.

\subsection{Moment coordinates and stitching}

\begin{lemma}\label{lem:moments}
For every point $s_0$ of a positive-dimensional stratum there are smooth real
functions $\phi_1,\ldots,\phi_d$ on $\PP^n$, $d=\dim_\R S$, such that
\[
 \Theta(s)=
 \left(\int\phi_1\,d\mu_s,\ldots,\int\phi_d\,d\mu_s\right)
\]
is a smooth real coordinate map near $s_0$.  On a smaller neighborhood,
\begin{equation}\label{eq:moment-comparison}
 |s-t|\le C W_2(\mu_s,\mu_t),
 \qquad
 \bar d(s,t)\le C|s-t|.
\end{equation}
\end{lemma}

\begin{proof}
By \eqref{eq:firstvariation}, the differentials of the moment functions are
$v\mapsto\int d\phi(N_v)\,d\mu_s$.  These differentials separate tangent vectors.  If they all vanish, then
$N_v$ is orthogonal to every ambient gradient, while \cref{prop:resolved-speed}
places $N_v$ in their closed span.  Orthogonality within this
closed span forces $N_v=0$; positivity of $g_S$ then forces $v=0$.  Choose $d$ of these moments with independent differentials; the inverse function
theorem makes them local coordinates.
Each $\phi_j$ is Lipschitz on compact $\PP^n$, and therefore
\[
 \left|\int\phi_j\,d\mu_s-\int\phi_j\,d\mu_t\right|
 \le \operatorname{Lip}(\phi_j)W_1(\mu_s,\mu_t)
 \le \operatorname{Lip}(\phi_j)W_2(\mu_s,\mu_t).
\]
The inverse coordinate map provides the first inequality in
\eqref{eq:moment-comparison}.  For the second, \cref{prop:stratum-distance}
implies $\bar d(s,t)\le d_{g_S}(s,t)$ for nearby points in the same stratum, and
the smooth local upper bound on $g_S$ implies $d_{g_S}(s,t)\le C|s-t|$.
\end{proof}

\begin{lemma}\label{lem:coincidence}
Let $\gamma_1,\gamma_2:I\to(Z,d)$ be Lipschitz curves into a metric space and let
$E\subset I$ be measurable.  If $\gamma_1=\gamma_2$ on $E$, then
\[
 |\dot\gamma_1|(t)=|\dot\gamma_2|(t)
\]
for almost every $t\in E$ for which the two metric derivatives exist.  The
conclusion also holds for ordinary derivatives of real Lipschitz functions.
\end{lemma}
\begin{proof}
For almost every $t\in E$, $t$ is a density point of $E$ and both metric derivatives
exist.  Fix one such $t$.  There is a sequence
$t_k\in E\setminus\{t\}$ with $t_k\to t$.  Since
$\gamma_1(t_k)=\gamma_2(t_k)$ and $\gamma_1(t)=\gamma_2(t)$,
\[
 \frac{d(\gamma_1(t_k),\gamma_1(t))}{|t_k-t|}
 =
 \frac{d(\gamma_2(t_k),\gamma_2(t))}{|t_k-t|}.
\]
Passing to the metric derivative limits proves the claim.  The argument for
real-valued Lipschitz functions is the same.
\end{proof}

On each positive-dimensional stratum, choose countably many moment charts
$\Theta_j:U_j\to B_j$ onto Euclidean balls $B_j\subset\R^d$, together with
compact cores $K_j\Subset U_j$ whose interiors cover the stratum and whose
images are contained in smaller concentric balls.  Thus, after shrinking if
necessary,
$\Theta_j(K_j)$ has positive Euclidean distance from
$\partial\Theta_j(U_j)=\partial B_j$.

\begin{lemma}\label{lem:stitch}
Let $f:[a,b]\to\R$ be continuous.  Suppose a countable Borel partition $(E_j)$ and
Lipschitz functions $f_j:[a,b]\to\R$ satisfy
\[
 f=f_j\ \text{on }E_j,
 \qquad
 |f_j'|\le L\ \text{a.e.\ on }E_j.
\]
Then $f$ is $L$-Lipschitz.
\end{lemma}

\begin{proof}
For every interval $I=[s,t]$, the one-dimensional area formula implies
\[
 \mathcal L^1(f_j(E_j\cap I))
 \le
 \int_{E_j\cap I}|f_j'|
 \le
 L\,\mathcal L^1(E_j\cap I).
\]
Because $f(I)=\bigcup_jf_j(E_j\cap I)$ and continuity makes $f(I)$ an interval,
\[
 |f(t)-f(s)|
 \le
 \mathcal L^1(f(I))
 \le
 L(t-s).
\]
See \cite[Chapter~2]{AFP}.
\end{proof}

\subsection{Globalization across strata}\label{sec:action-proof}

\begin{proof}[Proof of \cref{thm:full-action}]
The forward implication in \eqref{eq:ACequiv}, together with
$|\dot\mu_t|_{W_2}\le|\dot\gamma(t)|_{\bar d}$, follows from the
$1$-Lipschitz continuity of $\mathcal M$.  Conversely, suppose $\mu_\cdot$ is
absolutely continuous and set
\[
 \ell(t)=\int_0^t|\dot\mu_s|_{W_2}\,ds.
\]
If $\ell(T)=0$, then $\mu_t$ is constant on $[0,T]$.  Each fiber of
$\mathcal M$ is finite, while $\gamma([0,T])$ is connected, so $\gamma$ must be
constant as well.  Assume henceforth that $\ell(T)>0$.
Whenever $\ell$ is constant on an interval, the measure curve $\mu_t$ is constant
on that interval.  The fibers of
$\mathcal M$ are finite because cycle measures determine the Chow point and
normalization is finite; the connected image of that interval under the continuous
curve $\gamma$ must be a single point.  As $\ell$ is continuous and nondecreasing, every level set $\ell^{-1}(r)$ is a
closed interval (possibly a point); hence $\gamma$ is constant on every fiber of
$\ell$.  The curve therefore factors
uniquely through
$\widetilde\gamma:[0,\ell(T)]\to Y$ with
\[
 \gamma=\widetilde\gamma\circ\ell.
\]
It is continuous because the continuous surjection
$\ell:[0,T]\to[0,\ell(T)]$ is a quotient map (compact domain, Hausdorff target).
The associated measure curve
$\widetilde\mu=\mathcal M\circ\widetilde\gamma$ is the usual arclength
reparametrization of $\mu$ and hence $1$-Lipschitz.  It remains to prove the same
bound for $\widetilde\gamma$ in $(Y,\bar d)$.

On the positive-dimensional strata, use the compact cores of the moment coordinate charts
$K_j\Subset U_j$ chosen above and refine the sets
$\widetilde\gamma^{-1}(K_j)$ to a disjoint Borel family.  For each
zero-dimensional stratum $\{q\}$, add the Borel piece
$E_q=\widetilde\gamma^{-1}(q)$.  These sets form a countable Borel partition of
the parameter interval.

Fix $p\in B$ and put $f(t)=\bar d(p,\widetilde\gamma(t))$.  On a
positive-dimensional piece $E_j$, \eqref{eq:moment-comparison} makes $f|_{E_j}$
Lipschitz in the arclength parameter; choose a McShane extension $f_j$ to the full
parameter interval.  On $E_q$ take instead the constant function
$f_q\equiv\bar d(p,q)$, for which $f=f_q$ on $E_q$ and $f_q'=0$.

For each positive-dimensional piece $E_j$, the set $E_j$ need not contain an
interval, so $\widetilde\gamma$ cannot be treated locally as a curve staying in one
stratum.  Moment coordinates provide a substitute: they produce a genuine
stratum-valued comparison curve that agrees with $\widetilde\gamma$ on $E_j$.  Write
\[
 m_j(\tau)=
 \left(
  \int\phi_1\,d\widetilde\mu_\tau,\ldots,
  \int\phi_d\,d\widetilde\mu_\tau
 \right)
\]
and set
\[
 \delta_j=\dist\bigl(\Theta_j(K_j),\partial\Theta_j(U_j)\bigr)>0,
 \qquad
 O_j=\left\{\tau:\dist\bigl(m_j(\tau),\Theta_j(K_j)\bigr)<\delta_j/2\right\}.
\]
The set $O_j$ is open and contains $E_j$.  On $O_j$ define the comparison curve
\[
 \sigma_j(\tau)=\Theta_j^{-1}(m_j(\tau)).
\]
Its image stays in a compact subset of $U_j$ on which $\Theta_j^{-1}$ has bounded
derivative.  Because $m_j$ is Lipschitz, $\sigma_j$ is Lipschitz on each connected
component of $O_j$; its measure curve is Lipschitz as well by
\eqref{eq:moment-comparison} and $W_2\le\bar d$.  On $E_j$ we also have $\sigma_j=\widetilde\gamma$.  The open subset $O_j\subset\R$ has at
most countably many connected components.  On each such interval,
\cref{prop:stratum-distance} provides
\[
 \left|\frac d{d\tau}\bar d(p,\sigma_j(\tau))\right|
 \le
 \left|\frac d{d\tau}\mu_{\sigma_j(\tau)}\right|_{W_2}.
\]
On $E_j$ the comparison measure curve coincides with $\widetilde\mu$, which is
$1$-Lipschitz.  Applying \cref{lem:coincidence} on the countably many components
of $O_j$ shows
\[
 |\partial_\tau\mu_{\sigma_j}|_{W_2}
 =|\partial_\tau\widetilde\mu|_{W_2}\le1
 \quad\text{for a.e. }\tau\in E_j.
\]
The real-valued version of \cref{lem:coincidence}, applied to $f_j$ and
$\tau\mapsto\bar d(p,\sigma_j(\tau))$, shows that
$|f_j'|\le1$ for almost every $\tau\in E_j$.  Taking the countable union of
these exceptional sets preserves nullity.

The stitching lemma now makes $f$ globally $1$-Lipschitz.  To recover the metric
from these scalar estimates, choose a countable dense set $D_0\subset B$; it is
dense in $(Y,\bar d)$ because $B$ is dense in its completion.
For every $x,y\in Y$,
\[
 \bar d(x,y)=\sup_{p\in D_0}\bigl|\bar d(p,x)-\bar d(p,y)\bigr|,
\]
because the triangle inequality bounds the supremum by $\bar d(x,y)$, while
choosing points of $D_0$ converging to $x$ recovers the reverse inequality.  Applying
the preceding estimate for every $p\in D_0$ shows that
$\widetilde\gamma$ itself is $1$-Lipschitz for $\bar d$.  Returning to the original
parameter,
\[
 \bar d(\gamma(s),\gamma(t))
 \le
 \int_s^t|\dot\mu_r|_{W_2}\,dr.
\]
We have $|\dot\gamma|_{\bar d}\le|\dot\mu|_{W_2}$.  The reverse inequality comes from
the global $1$-Lipschitz property of $\mathcal M$, giving equality of the metric
derivatives almost everywhere.  Their $L^p$-integrability is equivalent for every
$1\le p\le\infty$, which proves both the speed identity and the $AC^p$
equivalence.
\end{proof}

\subsection{Consequences and recovery}

\begin{corollary}\label{cor:path-length}
For every continuous curve $\gamma:[a,b]\to Y$,
\[
 L_{\bar d}(\gamma)
 =L_{W_2}(\mathcal M\circ\gamma)
 \in[0,\infty],
\]
where $L_d$ denotes metric length.  Every continuous normalization path has the same length as its
Wasserstein image.
\end{corollary}
\begin{proof}
The global inequality $W_2\le\bar d$ implies
$L_{W_2}(\mathcal M\gamma)\le L_{\bar d}(\gamma)$.  Suppose the ambient length
is finite and let $\ell(t)$ be the accumulated $W_2$-length of
$\mathcal M\gamma$ on $[a,t]$.  If $\ell$ is constant on an interval, then
$\mathcal M\gamma$ is constant there.  The image of such an interval under
$\gamma$ is connected and lies in a finite fiber of $\mathcal M$, hence consists of
a single point.  The curve therefore factors as
$\gamma=\widetilde\gamma\circ\ell$ for a continuous curve
$\widetilde\gamma$ on $[0,L]$, where
$L=L_{W_2}(\mathcal M\gamma)$.  If $L=0$, this already forces $\gamma$ to be constant, and the conclusion is
immediate.  Otherwise the factor
$\widetilde\gamma$ is continuous because $\ell$ is a quotient map.  Its measure
curve is the arclength reparametrization of $\mathcal M\gamma$ and is $1$-Lipschitz.  By
\cref{thm:full-action},
$\widetilde\gamma$ is $1$-Lipschitz for $\bar d$ and the two metric speeds agree
almost everywhere.  Metric length is invariant under this continuous nondecreasing reparametrization,
so both lengths equal $L$.  If the ambient length is infinite, the first inequality
forces $L_{\bar d}(\gamma)=\infty$ as well.
\end{proof}

Recovery from the dense locus is a purely metric approximation problem, for which
we use the following lemma.
\begin{lemma}\label{lem:energy-recovery}
Let $(\overline Z,d)$ be the metric completion of a length space $(Z,d)$ and let
$\eta\in AC^2([0,1],\overline Z)$.  Then there are curves
$\eta_k\in AC^2([0,1],Z)$ such that
\begin{equation}\label{eq:energy-recovery}
 \sup_{t\in[0,1]}d(\eta_k(t),\eta(t))\longrightarrow0,
 \qquad
 \int_0^1|\dot\eta_k|^2\,dt
 \longrightarrow
 \int_0^1|\dot\eta|^2\,dt.
\end{equation}
In particular, $\eta_k(0)\to\eta(0)$ and $\eta_k(1)\to\eta(1)$.  If either
endpoint of $\eta$ belongs to $Z$, the corresponding endpoint of every
$\eta_k$ may be chosen equal to it.
\end{lemma}

\begin{proof}
For an $AC^2$ curve in a metric space, the quadratic energy admits the discrete
characterization
\begin{equation}\label{eq:discrete-energy}
 \int_0^1|\dot\eta|^2\,dt
 =\sup_{\mathcal P}\sum_i
 \frac{d(\eta(t_i),\eta(t_{i+1}))^2}{t_{i+1}-t_i},
\end{equation}
where the supremum runs over finite partitions
$\mathcal P=\{0=t_0<\cdots<t_N=1\}$; see, e.g., \cite[Chapter~1]{AGS}.
The discrete energy is nondecreasing under refinement.  Indeed, inserting
$r\in(s,t)$ gives
\[
  \frac{d(x,z)^2}{t-s}
  \le
  \frac{d(x,y)^2}{r-s}
  +\frac{d(y,z)^2}{t-r},
\]
by the triangle inequality and Cauchy--Schwarz.  Starting from partitions whose
discrete energies approach the supremum and taking successive common
refinements with partitions of mesh at most $1/k$, we obtain a refining sequence
$\mathcal P_k$ whose mesh tends to zero and whose discrete energies converge to
\eqref{eq:discrete-energy}.

For the nodes of $\mathcal P_k$, choose $z_i^k\in Z$ with
\[
 \max_i d(z_i^k,\eta(t_i^k))\le\delta_k,
\]
where $0<\delta_k\le k^{-1}$ is small enough that the perturbation of the
finite discrete-energy sum is at most $1/k$.  If $\eta(0)\in Z$, take
$z_0^k=\eta(0)$, and similarly take $z_N^k=\eta(1)$ whenever
$\eta(1)\in Z$.  This is possible because every partition has finitely many
nodes and $Z$ is dense in its completion.  As $Z$ is a length space, join
$z_i^k$ to $z_{i+1}^k$ by a curve in $Z$ of length at most
\[
 d(z_i^k,z_{i+1}^k)+\epsilon_{i,k},
\]
where the positive errors are chosen so that their contribution to the sum of
squared lengths divided by $t_{i+1}^k-t_i^k$ is at most $1/k$ and
$\max_i\epsilon_{i,k}\le1/k$.  Parametrize each connecting piece with constant
speed on its partition interval and concatenate.  Denote the resulting curve by
$\eta_k$.

Its energy is bounded above by the perturbed discrete energy, so
\[
 \limsup_{k\to\infty}\int_0^1|\dot\eta_k|^2\,dt
 \le \int_0^1|\dot\eta|^2\,dt.
\]
Let $\omega_\eta$ be the modulus of continuity of $\eta$ and
$|\mathcal P_k|$ the mesh.  On the $i$th partition interval,
\[
 \sup_{t\in[t_i^k,t_{i+1}^k]}d(\eta_k(t),\eta(t))
 \le 2\omega_\eta(|\mathcal P_k|)+3\delta_k+\epsilon_{i,k},
\]
so $\eta_k\to\eta$ uniformly.  For every fixed finite partition
$\mathcal P=\{0=s_0<\cdots<s_M=1\}$, uniform convergence and
\eqref{eq:discrete-energy} imply
\[
 \sum_{j=0}^{M-1}
 \frac{d(\eta(s_j),\eta(s_{j+1}))^2}{s_{j+1}-s_j}
 =\lim_{k\to\infty}
 \sum_{j=0}^{M-1}
 \frac{d(\eta_k(s_j),\eta_k(s_{j+1}))^2}{s_{j+1}-s_j}
 \le\liminf_{k\to\infty}\int_0^1|\dot\eta_k|^2\,dt.
\]
Taking the supremum over $\mathcal P$ gives the reverse energy inequality.  With
the upper bound already proved, this is \eqref{eq:energy-recovery}.  The endpoint
statements are built into the choice of nodes.
\end{proof}

\begin{corollary}\label{cor:action-distance}
For $y_0,y_1\in Y$,
\begin{equation}\label{eq:action-distance}
 \bar d(y_0,y_1)^2
 =
 \inf_{\substack{\gamma\in C([0,1],Y),\ \gamma(0)=y_0,\ \gamma(1)=y_1\\
                  \mathcal M\gamma\in AC^2(W_2)}}
 \int_0^1
 \left|\frac d{dt}\mathcal M(\gamma(t))\right|_{W_2}^2\,dt.
\end{equation}
The infimum is attained by a constant-speed $\bar d$ geodesic.  More generally,
every admissible curve $\gamma$ of finite action admits recovery curves
$\gamma_k\in AC^2([0,1],B)$ such that
\[
 \sup_{t\in[0,1]}\bar d(\gamma_k(t),\gamma(t))\to0
\]
and their ambient Wasserstein actions converge to that of $\mathcal M\gamma$.
Whenever $\gamma(0)\in B$ or $\gamma(1)\in B$, the corresponding endpoint can
be kept fixed for every $k$; in particular, both endpoints are preserved when
they belong to $B$.
\end{corollary}

\begin{proof}
By \cref{thm:full-action}, admissible curves and their ambient Wasserstein
actions are exactly the $AC^2$ curves and metric energies in $(Y,\bar d)$.  The
metric energy inequality
$\bar d(y_0,y_1)^2\le\int_0^1|\dot\gamma|_{\bar d}^2$ provides the lower bound in
\eqref{eq:action-distance}, while a constant-speed $\bar d$ geodesic attains
equality.

For the recovery statement, let $\gamma$ be any admissible curve of finite
action.  Then
$\gamma\in AC^2([0,1],Y)$ by \cref{thm:full-action}.  Apply
\cref{lem:energy-recovery} to the dense length subspace $(B,d_B)$ of its
completion $(Y,\bar d)$.  We obtain $\gamma_k\in AC^2([0,1],B)$ converging
uniformly to $\gamma$ and satisfying
\[
 \int_0^1|\dot\gamma_k|_{d_B}^2\,dt
 \longrightarrow
 \int_0^1|\dot\gamma|_{\bar d}^2\,dt.
\]
On $B$, \eqref{eq:speed} identifies the left-hand energies with the ambient
Wasserstein actions, while \eqref{eq:full-action-speed} identifies the limiting
energy with the action of $\mathcal M\gamma$.  The endpoint assertions are those
of \cref{lem:energy-recovery}.  Together with \cref{thm:full-action}, this
completes the proof of \cref{thm:intro-action}.
\end{proof}

\begin{corollary}
\label{cor:acl-identification}
Let $X=P_{n,d}$ and $B=P_{n,d}\setminus\Delta_{n,d}$.  With the convention
$\int_{\PP^1}\omega=1$ used in this paper, let $\bar d$ be the completed metric
of \cref{thm:main}.  Then the inner quadratic Wasserstein distance of
\cite{ACL} satisfies
\begin{equation}\label{eq:acl-exact-identification}
  W_{2,\mathrm{ACL}}^{\rm in}(p,q)
  =\sqrt{\pi}\,\bar d(p,q),
  \qquad p,q\in P_{n,d}.
\end{equation}
Equivalently, after rescaling the ambient Fubini--Study metric of \cite{ACL} so
that $\int_{\PP^1}\omega=1$, the two distances agree exactly.
\end{corollary}

\begin{proof}
Write $W_{2,0}$ for the ambient Wasserstein distance associated with the
Fubini--Study metric used in this paper, normalized by $\int_{\PP^1}\omega=1$.  The comparison in the proof of
\cref{thm:acl-compactness} shows that the normalized measure map is the same in
the two constructions, while
$W_{2,\mathrm{ACL}}=\sqrt{\pi}\,W_{2,0}$.  The admissible path classes
coincide and
\[
  W_{2,\mathrm{ACL}}^{\rm in}=\sqrt{\pi}\,W_{2,0}^{\rm in}.
\]
Here $W_{2,0}^{\rm in}$ is the constrained inner action distance defined in
\eqref{eq:w20-inner-definition}.  Under the specialization $X=P_{n,d}$, the
general measure map $\mathcal M$ of Section~5 is exactly $\mathcal M_{\rm hyp}$.
Because $P_{n,d}$ is smooth and normal,
\cref{cor:action-distance} applied to $X=P_{n,d}$ identifies this constrained
action distance with the completed metric:
\[
  \bigl(W_{2,0}^{\rm in}(p,q)\bigr)^2
  =
  \inf_{\substack{\gamma\in C([0,1],P_{n,d}),\
                    \gamma(0)=p,\ \gamma(1)=q\\
                    \mathcal M\gamma\in AC^2(W_{2,0})}}
  \int_0^1
  \left|\frac d{dt}\mathcal M(\gamma(t))\right|_{W_{2,0}}^2\,dt
  =
  \bar d(p,q)^2.
\]
The two identities together give \eqref{eq:acl-exact-identification}.
\end{proof}

\begin{corollary}\label{cor:admissible-locus-independent}
Let $B_1,B_2\subset X_{\rm reg}$ be two admissible dense smooth Zariski-open
subsets, and let $\bar d_1,\bar d_2$ be the completion metrics on $Y=X^\nu$
obtained from their resolved Wasserstein metrics.  Then
\[
  \bar d_1=\bar d_2.
\]
For the fixed Chow embedding and ambient Fubini--Study normalization, the
construction is intrinsic to $X$ and does not depend on the admissible
parameter locus chosen at the outset.
\end{corollary}
\begin{proof}
Apply \cref{cor:action-distance} to $B_1$ and $B_2$.  In both cases the squared
distance between $y_0,y_1\in Y$ is the infimum of
\[
  \int_0^1\left|\frac d{dt}\mathcal M(\gamma(t))\right|_{W_2}^2\,dt
\]
over the same class of continuous curves $\gamma:[0,1]\to Y$ with fixed
endpoints and $\mathcal M\gamma\in AC^2(W_2)$.  The right-hand side depends only on
$\mathcal M:Y\to\mathcal P_2(\PP^n)$, so the two metrics coincide.  This
completes the deferred locus-independence assertion in \cref{thm:main}.
\end{proof}

\section{Elliptic quartic boundary branches}\label{sec:quartic}
We conclude with elliptic quartics, where the abstract branch separation in
\cref{thm:main} becomes explicit through the Hilbert--Chow fiber.

Let $B_{2,2}$ be the locus of smooth complete intersections of two quadrics in
$\PP^3$ and set
\[
 X_{2,2}:=\overline{B_{2,2}}^{\,\Chow}_{\rm red}.
\]
Let
\[
   H=\Hilb^{4t}_1(\PP^3)
\]
be the principal Hilbert component containing smooth elliptic quartics.  The
classical compactification and blow-up description go back to
Avritzer--Vainsencher \cite{AV92}; a modern description as a double blow-up with smooth
centers appears in \cite{GLS}.  The Hilbert--Chow morphism restricts to a proper
birational morphism
\[
   q:H\longrightarrow X_{2,2},
\]
because it is an isomorphism over $B_{2,2}$ and its image is the reduced Chow
closure of that locus.

Let $C\subset \Pi\simeq\PP^2\subset\PP^3$ be an irreducible plane quartic.
Choose a linear equation $x=0$ for $\Pi$ and write
$\mathcal I_C=(x,F)$.  Here $I_\bullet$ denotes homogeneous or local defining
ideals, and $\mathcal I_\bullet$ the corresponding ideal sheaves.  The boundary
classification is taken from \cite[Theorem~5.2]{AV92}; for the type~(4) stratum
we use the coordinate form in \cite[Proposition~3.6]{CKSZ}.

Define the intrinsic length-two singularity locus
\begin{equation}\label{eq:quartic-ZC}
  \mathfrak Z_C
  :=\left\{Z\in\Hilb^2(\Pi):
  F\in H^0\bigl(\Pi,\mathcal I_{Z,\Pi}^{\,2}(4)\bigr)\right\}.
\end{equation}
For $Z\in\mathfrak Z_C$, put
$\mathcal J_Z=\mathcal I_{Z,\PP^3}$ and define the subscheme $Y_Z$ intrinsically by
\begin{equation}\label{eq:quartic-ideal-intrinsic}
   \mathcal I_{Y_Z}:=\mathcal I_C\cap\mathcal J_Z^2.
\end{equation}
\begin{lemma}\label{lem:quartic-hilbert}
Assume $F\in H^0(\Pi,\mathcal I_{Z,\Pi}^{\,2}(4))$.  There is a lift
$\widetilde F\in H^0(\PP^3,\mathcal J_Z^2(4))$ with
$\widetilde F|_\Pi=F$, and for every such lift
\begin{equation}\label{eq:quartic-ideal}
   \mathcal I_{Y_Z}=x\mathcal J_Z+(\widetilde F).
\end{equation}
The right-hand side is independent of the chosen lift, and
\begin{equation}\label{eq:quartic-quotient}
   \mathcal I_C/\mathcal I_{Y_Z}\simeq\mathcal O_Z(-1),
   \qquad P_{Y_Z}(t)=4t.
\end{equation}
Moreover,
\begin{equation}\label{eq:quartic-annihilator}
  \operatorname{Ann}_{\mathcal O_{\PP^3}}
  \bigl(\mathcal I_C/\mathcal I_{Y_Z}\bigr)=\mathcal J_Z,
\end{equation}
so the quotient recovers the entire length-two subscheme $Z$, not merely its
support.
\end{lemma}
\begin{proof}
Every length-two subscheme of the plane is a complete intersection of a line and a
quadric.  Write $\mathcal I_{Z,\Pi}=(y,q)$, with $\deg y=1$ and $\deg q=2$.
The square of this complete intersection ideal has the standard resolution
\[
0\longrightarrow
\mathcal O_\Pi(-4)\oplus\mathcal O_\Pi(-5)
\longrightarrow
\mathcal O_\Pi(-2)\oplus\mathcal O_\Pi(-3)\oplus\mathcal O_\Pi(-4)
\longrightarrow\mathcal I_{Z,\Pi}^{\,2}\longrightarrow0.
\]
After twisting by $4$, the vanishings
$H^1(\Pi,\mathcal O_\Pi)=H^1(\Pi,\mathcal O_\Pi(-1))=0$ show that the
global sections generated by $y^2,yq,q^2$ surject onto
$H^0(\Pi,\mathcal I_{Z,\Pi}^{\,2}(4))$.  We may therefore write
\[
   F=y^2Q+yq\ell+cq^2
\]
for a quadric $Q$, a linear form $\ell$, and a constant $c$.  Lift these
factors to $\PP^3$.  The resulting quartic
$\widetilde F$ lies in $(x,y,q)^2=\mathcal J_Z^2$ and restricts to $F$ on the
plane.

Locally, $\mathcal J_Z$ is generated by a regular sequence containing $x$.  In particular,
\begin{equation}\label{eq:colon-square}
   (\mathcal J_Z^2:x)=\mathcal J_Z;
\end{equation}
for instance, this follows from the polynomial associated graded algebra of an
ideal generated by a regular sequence.  The inclusion
$x\mathcal J_Z+(\widetilde F)\subset\mathcal I_C\cap\mathcal J_Z^2$ is immediate.
For the reverse inclusion, write a local section of the intersection as
$xh+a\widetilde F$.
Both this section and $\widetilde F$ lie in $\mathcal J_Z^2$, hence
$xh\in\mathcal J_Z^2$, and \eqref{eq:colon-square} forces
$h\in\mathcal J_Z$.  This establishes \eqref{eq:quartic-ideal}.  If two allowed lifts
differ by $xA$, then $xA\in\mathcal J_Z^2$, so
$A\in\mathcal J_Z$ by \eqref{eq:colon-square}; the two lifts define the same
ideal.

Multiplication by $x$ induces a surjection
$\mathcal O_{\PP^3}(-1)\to\mathcal I_C/\mathcal I_{Y_Z}$.  Its kernel is
$\mathcal J_Z(-1)$: if
$xh=xg+a\widetilde F$ with $g\in\mathcal J_Z$, reduction modulo $x$ shows
$(a\bmod x)F=0$ in the integral coordinate ring of the plane, so $a=xb$ and
$h=g+b\widetilde F\in\mathcal J_Z$.  The quotient is
$\mathcal O_Z(-1)$.  Its annihilator as an $\mathcal O_{\PP^3}$-module is
precisely $\mathcal J_Z$, which proves \eqref{eq:quartic-annihilator}.  A plane
quartic has Hilbert polynomial $4t-2$, while $Z$ has length two, and
\eqref{eq:quartic-quotient} implies $P_{Y_Z}(t)=4t$.
\end{proof}

\begin{lemma}\label{lem:quartic-degeneration}
For every $Z\in\mathfrak Z_C$, the scheme $Y_Z$ belongs to the principal Hilbert
component $H$ of elliptic quartics, and its fundamental cycle is $C$.
\end{lemma}
\begin{proof}
Fix $Z\in\mathfrak Z_C$ and choose plane coordinates with
$I_{Z,\Pi}=(y,q)$, where $q$ is a quadric not divisible by $y$.  Write
\[
 F=cq^2+yq\ell+y^2Q.
\]
Irreducibility of $C$ forces $c\ne0$.  Using the same letters for lifts independent
of $x$, introduce a parameter $\tau$ and set
\[
 A_\tau=cx^2-\tau\ell x+\tau^2Q,
 \qquad
 B_\tau=xy+\tau q,
 \qquad
 C_\tau=cxq-\tau(yQ+\ell q).
\]
Over the fraction field $K=\C(\tau)$, the polynomial $B_\tau$ is linear in $x$
and primitive, with coprime coefficients $y$ and $\tau q$; Gauss's lemma implies
its irreducibility.  In the fraction field of the plane coordinate ring, imposing $B_\tau=0$ forces
$x=-\tau q/y$, so
\[
 A_\tau\big|_{B_\tau=0}
 =\frac{\tau^2}{y^2}\bigl(cq^2+yq\ell+y^2Q\bigr)
 =\frac{\tau^2F}{y^2}\ne0.
\]
$B_\tau$ cannot divide $A_\tau$.  The two quadrics have no common factor
and form a regular sequence.  Let $U\subset H^0(\PP^3,\mathcal O(2))^{\oplus2}$ be the locus of ordered
pairs forming a regular sequence.  It is a nonempty irreducible Zariski-open
set, and its smooth complete intersection locus $U^{\rm sm}$ is nonempty and
dense.  The relative Koszul resolution shows that the universal
complete intersection over $U$ is flat and defines a morphism
$\varphi:U\to\Hilb^{4t}(\PP^3)$.  The smooth locus maps into the principal component $H$.  Because $U^{\rm sm}$
is dense in $U$ and $H$ is closed, the whole image $\varphi(U)$ lies in $H$.  In
particular the generic pair $(A_\tau,B_\tau)$ lies on the principal component
$H$.

For the special fiber, pass to the DVR
$\Lambda=\C[\tau]_{(\tau)}$ and let
\[
 \mathfrak I=((A_\tau,B_\tau):\tau^\infty)
 \subset \Lambda[x_0,x_1,x_2,x_3]
\]
be the $\tau$-saturated homogeneous ideal.  The identities
\[
 yA_\tau-cxB_\tau+\tau\ell B_\tau=-\tau C_\tau,
 \qquad
 yC_\tau-cqB_\tau=-\tau F
\]
show that $C_\tau$ and $F$ belong to $\mathfrak I$.  The quotient
$\Lambda[x_0,x_1,x_2,x_3]/\mathfrak I$ has no $\tau$-torsion.  Each graded piece
is therefore a finitely generated torsion-free module over the DVR $\Lambda$,
hence free.  The associated projective family is flat over $\Lambda$.  Its generic fiber is the complete intersection above,
while its special fiber has Hilbert polynomial $4t$ and homogeneous ideal containing
\[
 I_{Y_Z}:=(x^2,xy,xq,F).
\]
By \cref{lem:quartic-hilbert}, the sheafification of $I_{Y_Z}$ is precisely
$\mathcal I_{Y_Z}$.  The special fiber is a closed subscheme of $Y_Z$ with the same Hilbert polynomial.
The quotient coherent sheaf therefore has zero Hilbert polynomial and vanishes;
the special fiber is exactly $Y_Z$.  The flat family defines a morphism $\operatorname{Spec}\Lambda$ to the Hilbert
scheme.  Since its generic point lies in the closed component $H$, so does the
special point $Y_Z$.  Finally,
$\mathcal I_C/\mathcal I_{Y_Z}\simeq\mathcal O_Z(-1)$ by
\cref{lem:quartic-hilbert}, and this quotient is supported in dimension zero.
Hence $Y_Z$ and $C$ agree at the generic point of every one-dimensional component,
so the fundamental cycle of $Y_Z$ is $C$.
\end{proof}

\begin{lemma}\label{lem:no-punctual}
If $p$ is an ordinary node of $C$, there is no length-two scheme $Z$ supported at $p$ such that $F\in I_{Z,\Pi}^2$.
\end{lemma}
\begin{proof}
Every punctual length-two subscheme of the smooth surface $\Pi$ is curvilinear.  In
formal coordinates we may write $I_Z=(u,v^2)$.  The quadratic part of $I_Z^2$ is
then one-dimensional, generated by $u^2$.  If $F\in I_Z^2$, the quadratic tangent
cone of $C$ at $p$ would have rank one, contradicting the fact that an ordinary node has a nondegenerate quadratic tangent cone of rank two.
\end{proof}

\begin{theorem}[Hilbert--Chow fibers of plane quartics]\label{thm:quartic-fiber}
Let $C\subset\Pi$ be an irreducible plane quartic and let $\mathfrak Z_C$ be
\eqref{eq:quartic-ZC}.  Every $Z\in\mathfrak Z_C$ determines a point
$[Y_Z]\in H$ with fundamental cycle $C$, and every point of the set-theoretic fiber
over $[C]$ arises in this way.  Whenever
$[C]\in X_{2,2}$, the map
\begin{equation}\label{eq:quartic-fiber-bijection}
  \mathfrak Z_C\longrightarrow q^{-1}([C]),
  \qquad Z\longmapsto[Y_Z],
\end{equation}
is a bijection.
\end{theorem}
\begin{proof}
The forward implication is \cref{lem:quartic-degeneration}.  Conversely, let
$[Y]\in q^{-1}([C])$.  By the boundary classification of
\cite[Theorem~5.2]{AV92}, a Hilbert point with reduced irreducible plane quartic
fundamental cycle lies in the type~(4) stratum.  In the coordinate notation fixed
above its ideal has the form
\begin{equation}\label{eq:CKSZ-type4}
 I_Y=(x^2,xy,xq_1,\,y^2q_2+yq_1\ell+cq_1^2).
\end{equation}
The normal forms (1)--(3) in that classification cannot have a reduced
irreducible plane quartic fundamental cycle.  Put
\[
 G=y^2q_2+yq_1\ell+cq_1^2.
\]
Because $[Y]$ maps to $[C]$, restriction to the plane gives
$G|_\Pi=\lambda F$ for some $\lambda\ne0$.  Multiplying the last generator by
$\lambda^{-1}$ does not change $I_Y$.  Set
$\widetilde F=\lambda^{-1}G$ and $\bar q_1=q_1|_\Pi$.  Then
\[
 \widetilde F|_\Pi=F,
 \qquad
 I_{Z,\Pi}=(y,\bar q_1).
\]
Irreducibility of $C$ implies that $\bar q_1$ is not divisible by $y$, so
$Z=V_\Pi(y,\bar q_1)$ is a length-two complete intersection.  Since
$\widetilde F\in\mathcal J_Z^2(4)$, one has
$F\in I_{Z,\Pi}^2$, so $Z\in\mathfrak Z_C$.  Sheafifying the type~(4) ideal
and applying \cref{lem:quartic-hilbert} identifies
\[
  \mathcal I_Y=x\mathcal J_Z+(\widetilde F)
  =\mathcal I_C\cap\mathcal J_Z^2
  =\mathcal I_{Y_Z}.
\]
Every point of the fiber therefore has the form $[Y_Z]$ for some
$Z\in\mathfrak Z_C$.

Uniqueness follows from \eqref{eq:quartic-annihilator}:
\[
  \mathcal J_Z
  =\operatorname{Ann}_{\mathcal O_{\PP^3}}
    (\mathcal I_C/\mathcal I_{Y_Z}),
\]
so the Hilbert point $Y_Z$ recovers the full length-two subscheme $Z$, including
its nonreduced structure.  The subscheme $Z$ is therefore uniquely determined by
$Y_Z$, proving injectivity and \eqref{eq:quartic-fiber-bijection}.
\end{proof}

\begin{corollary}\label{cor:nodal-quartic-fiber}
Assume that the singular locus of $C$ consists of exactly $\delta\ge2$ ordinary
nodes
\[
  \Sing C=\{p_1,\ldots,p_\delta\}
\]
and no other singularities.  Then
\begin{equation}\label{eq:nodal-ZC}
  \mathfrak Z_C
  =\{p_i+p_j:1\le i<j\le\delta\},
\end{equation}
so, in particular, $[C]\in X_{2,2}$ and
\begin{equation}\label{eq:nodal-fiber-count}
  \#q^{-1}([C])=\binom{\delta}{2}.
\end{equation}
Moreover, the normalization $X_{2,2}^\nu$ has exactly
$\binom{\delta}{2}$ points over $[C]$.
\end{corollary}
\begin{proof}
If $Z\in\mathfrak Z_C$, then every point of $\supp Z$ is singular on $C$, since
$F\in I_{Z,\Pi}^2$ forces the first jet of $F$ to vanish along the support.
By \cref{lem:no-punctual}, $Z$ cannot be supported at a single ordinary node.
It must be reduced and supported at two distinct nodes, say
$Z=p_i+p_j$ with $i<j$.  Conversely, for distinct nodes $p_i,p_j$, the
length-two reduced scheme $Z_{ij}=p_i+p_j$ satisfies
\[
  F\in I_{p_i,\Pi}^2\cap I_{p_j,\Pi}^2
   =I_{Z_{ij},\Pi}^2,
\]
where the equality is local on the disjoint support.  This proves
\eqref{eq:nodal-ZC}.  By
\cref{lem:quartic-degeneration,thm:quartic-fiber}, the point $[C]$ lies in
$X_{2,2}$ and the Hilbert--Chow fiber has the stated cardinality.

By \cite[Theorem~B]{GLS}, the principal component $H$ is smooth and, in particular,
normal, while $q:H\to X_{2,2}$ is proper birational.  In its Stein factorization
\[
 H\longrightarrow \widetilde X\longrightarrow X_{2,2},
\]
the first morphism has connected fibers and the second is finite; see
\cite[\href{https://stacks.math.columbia.edu/tag/03H0}{Tag~03H0}]{Stacks}.  Normality of $H$ and birationality of $q$ identify
$\widetilde X$ with $X_{2,2}^\nu$.  The finite fiber in
\eqref{eq:nodal-fiber-count} consists of isolated points, so each point is a
connected component of the $q$-fiber.  The normalization fiber has
exactly $\binom{\delta}{2}$ points.
\end{proof}

\begin{corollary}\label{cor:quartic-metric}
If $C$ has exactly three ordinary nodes $p_1,p_2,p_3$ and no other
singularities, write $Y_{ij}:=Y_{p_i+p_j}$.  Then
\[
  q^{-1}([C])=\{Y_{12},Y_{13},Y_{23}\},
\]
the normalization of $X_{2,2}$ has exactly three points over $[C]$, and the
Wasserstein completion of the smooth $(2,2)$ complete intersection locus has
exactly three boundary points above the same Chow cycle.
\end{corollary}
\begin{proof}
Take $\delta=3$ in \cref{cor:nodal-quartic-fiber} and apply
\cref{thm:main} to the smooth complete intersection locus.
\end{proof}


\begin{thebibliography}{99}

\bibitem{AFP}
L.~Ambrosio, N.~Fusco, and D.~Pallara,
\emph{Functions of Bounded Variation and Free Discontinuity Problems},
Oxford Mathematical Monographs, Oxford University Press, Oxford, 2000.
\href{https://doi.org/10.1093/oso/9780198502456.001.0001}{doi:10.1093/oso/9780198502456.001.0001}.

\bibitem{AGS}
L.~Ambrosio, N.~Gigli, and G.~Savar\'e,
\emph{Gradient Flows in Metric Spaces and in the Space of Probability Measures},
2nd ed., Lectures in Mathematics ETH Z\"urich, Birkh\"auser, Basel, 2008.
\href{https://doi.org/10.1007/978-3-7643-8722-8}{doi:10.1007/978-3-7643-8722-8}.

\bibitem{ACL}
P.~Antonini, F.~Cavalletti, and A.~Lerario,
\emph{Optimal transport between algebraic hypersurfaces},
Geom. Funct. Anal. \textbf{35} (2025), 43--112.
\href{https://doi.org/10.1007/s00039-025-00699-w}{doi:10.1007/s00039-025-00699-w}.

\bibitem{AV92}
D.~Avritzer and I.~Vainsencher,
\emph{Compactifying the space of elliptic quartic curves},
in \emph{Complex Projective Geometry} (Trieste/Bergen, 1989),
London Math. Soc. Lecture Note Ser. \textbf{179}, Cambridge University Press,
Cambridge, 1992, pp.~47--58.
\href{https://doi.org/10.1017/CBO9780511662652.005}{doi:10.1017/CBO9780511662652.005}.

\bibitem{AS06}
R.~Axelsson and G.~Schumacher,
\emph{K\"ahler geometry of Douady spaces},
Manuscripta Math. \textbf{121} (2006), no.~3, 277--291.
\href{https://doi.org/10.1007/s00229-006-0020-z}{doi:10.1007/s00229-006-0020-z}.

\bibitem{AS21}
R.~Axelsson and G.~Schumacher,
\emph{The Weil--Petersson current on Douady spaces},
Math. Nachr. \textbf{294} (2021), no.~4, 638--656.
\href{https://doi.org/10.1002/mana.201900126}{doi:10.1002/mana.201900126}.

\bibitem{Barlet75}
D.~Barlet,
\emph{Espace analytique r\'eduit des cycles analytiques complexes compacts d'un
espace analytique complexe de dimension finie},
in \emph{Fonctions de plusieurs variables complexes II},
Lecture Notes in Math. \textbf{482}, Springer, Berlin, 1975, pp.~1--158.

\bibitem{BarletVarouchas89}
D.~Barlet and J.~Varouchas,
\emph{Fonctions holomorphes sur l'espace des cycles},
Bull. Soc. Math. France \textbf{117} (1989), no.~3, 327--341.
\href{https://doi.org/10.24033/bsmf.2126}{doi:10.24033/bsmf.2126}.

\bibitem{BM25}
D.~Barlet and J.~Magn\'usson,
\emph{Complex Analytic Cycles II: The Cycle Space},
Grundlehren Math. Wiss. \textbf{363}, Springer, Cham, 2025.
\href{https://doi.org/10.1007/978-3-031-84845-2}{doi:10.1007/978-3-031-84845-2}.

\bibitem{BB00}
J.-D.~Benamou and Y.~Brenier,
\emph{A computational fluid mechanics solution to the Monge--Kantorovich mass
transfer problem},
Numer. Math. \textbf{84} (2000), no.~3, 375--393.
\href{https://doi.org/10.1007/s002110050002}{doi:10.1007/s002110050002}.

\bibitem{CKSZ}
I.~Cheltsov, A.-S.~Kaloghiros, R.~\'Smiech, and J.~Zhao,
\emph{Degenerations of elliptic quartics and K-moduli of Fano threefolds},
\href{https://arxiv.org/abs/2608.30855}{arXiv:2608.30855v1} (2026).

\bibitem{DK01}
J.-P.~Demailly and J.~Koll\'ar,
\emph{Semi-continuity of complex singularity exponents and K\"ahler--Einstein
metrics on Fano orbifolds},
Ann. Sci. \'Ec. Norm. Sup\'er. (4) \textbf{34} (2001), no.~4, 525--556.
\href{https://doi.org/10.1016/S0012-9593(01)01069-2}{doi:10.1016/S0012-9593(01)01069-2}.

\bibitem{DL}
A.~Di Lorenzo,
\emph{Intersection theory on moduli of smooth complete intersections},
Math. Z. \textbf{304} (2023), Art.~39.
\href{https://doi.org/10.1007/s00209-023-03299-2}{doi:10.1007/s00209-023-03299-2}.

\bibitem{Fujiki78}
A.~Fujiki,
\emph{Closedness of the Douady spaces of compact K\"ahler spaces},
Publ. Res. Inst. Math. Sci. \textbf{14} (1978), no.~1, 1--52.
\href{https://doi.org/10.2977/PRIMS/1195189279}{doi:10.2977/PRIMS/1195189279}.

\bibitem{GLS}
P.~Gallardo, C.~Lozano Huerta, and B.~Schmidt,
\emph{Families of elliptic curves in $\PP^3$ and Bridgeland stability},
Michigan Math. J. \textbf{67} (2018), no.~4, 787--813.
\href{https://doi.org/10.1307/mmj/1538705132}{doi:10.1307/mmj/1538705132}.

\bibitem{GKZ}
I.~M.~Gelfand, M.~M.~Kapranov, and A.~V.~Zelevinsky,
\emph{Discriminants, Resultants, and Multidimensional Determinants},
Birkh\"auser, Boston, 1994.
\href{https://doi.org/10.1007/978-0-8176-4771-1}{doi:10.1007/978-0-8176-4771-1}.

\bibitem{GGZ}
V.~Guedj, H.~Guenancia, and A.~Zeriahi,
\emph{Diameter of K\"ahler currents},
J. Reine Angew. Math. \textbf{820} (2025), 115--152.
\href{https://doi.org/10.1515/crelle-2024-0092}{doi:10.1515/crelle-2024-0092}.

\bibitem{LerarioFoCM26}
A.~Lerario,
\emph{Wasserstein geometry of analytic cycle spaces},
conference lecture, Foundations of Computational Mathematics (FoCM 2026),
Vienna, 2026; based on joint work with P.~Antonini, F.~Cavalletti, and L.~Cecchi;
\href{https://focm2026.univie.ac.at/?page_id=1822}{conference program}.

\bibitem{Lott08}
J.~Lott,
\emph{Some geometric calculations on Wasserstein space},
Comm. Math. Phys. \textbf{277} (2008), no.~2, 423--437.
\href{https://doi.org/10.1007/s00220-007-0367-3}{doi:10.1007/s00220-007-0367-3}.

\bibitem{Otto01}
F.~Otto,
\emph{The geometry of dissipative evolution equations: the porous medium equation},
Comm. Partial Differential Equations \textbf{26} (2001), nos.~1--2, 101--174.
\href{https://doi.org/10.1081/PDE-100002243}{doi:10.1081/PDE-100002243}.

\bibitem{PS02}
D.~H.~Phong and J.~Sturm,
\emph{Stability, energy functionals, and K\"ahler--Einstein metrics},
Comm. Anal. Geom. \textbf{11} (2003), no.~3, 565--597.
\href{https://doi.org/10.4310/CAG.2003.v11.n3.a6}{doi:10.4310/CAG.2003.v11.n3.a6}.

\bibitem{PS25}
E.~Pratt and B.~Sturmfels,
\emph{The Chow--Lam form},
J. Symbolic Comput. \textbf{131} (2025), Art.~102450.
\href{https://doi.org/10.1016/j.jsc.2025.102450}{doi:10.1016/j.jsc.2025.102450}.

\bibitem{Rydh08}
D.~Rydh,
\emph{Families of cycles and the Chow scheme},
Ph.D. thesis, KTH Royal Institute of Technology, Stockholm, 2008,
Trita-MAT. MA 08-MA-06;
\href{https://people.kth.se/~dary/thesis/}{available online}.

\bibitem{Stacks}
The Stacks Project Authors,
\emph{Stacks Project}, 2018,
\href{https://stacks.math.columbia.edu}{stacks.math.columbia.edu}.

\end{thebibliography}
\end{document}